\documentclass{article}
\title{Aspects of Predicative Algebraic Set Theory I: Exact Completion (DRAFT)}
\author{Benno van den Berg \& Ieke Moerdijk}
\date{September 24, 2007}
\usepackage{ast2}
\begin{document}
\maketitle


\section{Introduction}

This is the first in a series of three papers on Algebraic Set Theory.
Its main purpose is to lay the necessary groundwork for the next two
parts, one on Realisability \cite{bergmoerdijk07c} and the other on Sheaf Models in Algebraic Set Theory \cite{bergmoerdijk07d}.

Sheaf theory and realisability have been effective methods for constructing
models of various constructive and intuitionistic type theories \cite{hyland82, lambekscott86,sheaves753}. In particular, toposes constructed using sheaves or realisability provide models for intuitionistic higher order logic ({\bf HAH}), and it was shown by Freyd, Fourman, Friedman respectively by McCarthy in the 1980s that from these toposes one can construct models of intuitionistic Zermelo-Fraenkel set theory {\bf IZF} \cite{freyd80, fourman80, friedman73, mccarty84}. These constructions were non-elementary, in the technical sense that they used the class of all ordinal numbers external to the topos, i.e., ordinals in an ambient classical metatheory. The original purpose of Algebraic Set Theory \cite{joyalmoerdijk95} was to provide an elementary, categorical framework making such constructions of models of {\bf IZF} possible. More precisely, in \emph{loc.~cit.}~the authors proposed a notion of ``category with small maps'', which is a pair consisting of a category \ct{E} which behaves to some extent like a topos, and a class \smallmap{S} of arrows in \ct{E}, the ``small maps'', to be thought of as maps whose fibres are small in some \emph{a priori} given sense. It was proved that such a pair $(\ct{E}, \smallmap{S})$ always contains a special object $V$ (an initial ZF-algebra in the terminology of \cite{joyalmoerdijk95}), which is a model of {\bf IZF}. Although this was never proved in detail, the idea behind the definition of such pairs $(\ct{E}, \smallmap{S})$ was that they would be closed under sheaves and realisability. For example, for sheaves, this means that for any internal small site \ct{C} in $(\ct{E}, \smallmap{S})$, the category $\shct{\ct{E}}{\ct{C}}$ of internal sheaves is equipped with a natural class of maps $\smallmap{S}[\ct{C}]$,
for which the pair $(\shct{\ct{E}}{\ct{C}}, \smallmap{S}[\ct{C}])$ again satisfies the axioms for a ``category with small maps''. As a consequence, one would be able to apply and iterate sheaf and/or realisability constructions to obtain new categories with small maps from old ones, each of which contains a model of set theory $V$. The original constructions of Freyd, Fourman and McCarthy \cite{freyd80, fourman80, friedman73, mccarty84, kouwenhovenvanoosten05} form a special case of this. An immediate result would be that known independence proofs for {\bf HAH}, proved using topos-theoretic
techniques, can be transferred to {\bf IZF} (for example, \cite{fourmanhyland79,blassscedrov89,fourmanscedrov82}).

Subsequently, various alternative axiomatisations of categories with
small maps have been proposed, notably the one by Awodey, Butz, Simpson and Streicher \cite{awodeyetal04}. In particular, Simpson in \cite{simpson99} proves that {\bf IZF} is complete with respect to models in his axiomatisation of a category with small maps.

The main goal of this series of three papers is to investigate how these
techniques apply in the context of predicative type theories in the
style of Martin-L\"of \cite{martinlof84} and related predicative set theories such as Aczel's {\bf CZF} \cite{aczel78,aczelrathjen01}. A distinguishing feature of these type theories is that they do not allow power object constructions, but do contain inductive types (so-called ``W-types'') instead.  In analogy with the non-predicative case, we aim to find axioms for a suitable notion of ``category with a class of small maps'' $(\ct{E},\smallmap{S})$ where the category \ct{E} is some sort of predicative analogue of a topos, having equally good closure properties as in the impredicative case. In particular, the following should hold:
\begin{itemize}
\item[(i)] Any such pair $(\ct{E},\smallmap{S})$ contains an object $V$ which models {\bf CZF}.
\item[(ii)] The notion is closed under taking sheaves; i.e., for a internal site
\ct{C} (possibly satisfying some smallness conditions), the category of
internal sheaves in \ct{E} contains a class of small maps, so that we obtain a
similar such pair $(\shct{\ct{E}}{\ct{C}}, \smallmap{S}[\ct{C}])$.
\item[(iii)] The notion is closed under realisability: i.e., for any partial
combinatorial algebra \pca{A} in \ct{E}, one can construct a category $\Eff_{\ct{E}}[\pca{A}]$ of \pca{A}-effective objects (analogous to the effective topos \cite{hyland82}), and a corresponding class of small maps $\smallmap{S}[\pca{A}]$, so that the pair $(\Eff_{\ct{E}}[\pca{A}],\smallmap{S}[\pca{A}])$ again satisfies the axioms.
\item[(iv)] The notion admits a completeness theorem for {\bf CZF}, analogous to the one for {\bf IZF} mentioned above.
\end{itemize}
This list describes our goals for this series of papers, but is not
exhaustive. There are other constructions that are known to have useful
applications in the impredicative context of topos theory, {\bf HAH} and {\bf IZF}, which one might ask our predicative notion of categories with small maps
to be closed under, such as glueing and the construction of the category of
coalgebras for a (suitable) comonad \cite{wraith74,freydfriedmanscedrov87,lambekscott86}.

To reach these goals, one needs the category \ct{E} to have some exactness
properties, in particular to be closed under quotients of certain
equivalence relations. Indeed, some particular such quotients are needed
in (i) above to construct the model $V$ as a quotient of a certain universal
W-type, and in (ii) to construct the associated sheaf functor. On the
other hand, the known methods of proof to achieve the goals (iii) and (iv)
naturally give rise to pairs $(\ct{E}, \smallmap{S})$ for which \ct{E} is not sufficiently exact. In order to overcome this difficulty, we identify the precise degree of exactness which is needed, and prove that for the kinds of categories with a class of small maps $(\ct{E}, \smallmap{S})$ which arise in (iii) and (iv), one can construct a good ``exact completion''  $(\overline{\ct{E}}, \overline{\smallmap{S}})$.  The first of these three
papers is mainly concerned with analysing this exact completion.

To illustrate the work involved, let us consider the axiom of Subset
collection of {\bf CZF}, which can be formulated in the tradional form
\begin{description}
\item[Subset collection:] $\exists c \, \forall z \, ( \forall x \epsilon a \, \exists y \epsilon b \, \phi(x,y,z) \rightarrow \exists d \epsilon c \, \mbox{B}(x \epsilon a, y \epsilon d) \, \phi(x, y, z)) $.
\end{description}
or in terms of multi-valued functions as what is called the Fullness axiom (see Section 3.7 below)
\begin{description}
\item[Fullness:] $\exists z\;\! (z \;\!\subseteq \;\!{\bf mvf}(a,b) \land \forall x \;\!\epsilon \;\!{\bf mvf}(a,b) \;\! \exists c \;\!\epsilon \;\!z \;\! (c \;\!\subseteq \;\!x))$,
\end{description}
where we have used the abbreviation ${\bf mvf}(a, b)$ for the class of multi-valued functions from $a$ to $b$, i.e., sets $r \subseteq a \times b$ such that $\forall x \epsilon a \, \exists y \epsilon b \, (x, y) \epsilon r$.
This Fullness axiom  has a categorical counterpart {\bf (F)}. This latter axiom is one of the axioms for our pairs $(\ct{E},\smallmap{S})$, for which we prove the following:
\begin{itemize}
\item[(a)] If $(\ct{E},\smallmap{S})$ satisfies {\bf (F)}, then the model $V$ constructed as in (i) satisfies Subset collection (see \refcoro{complforCZF} below).
\item[(b)] If $(\ct{E},\smallmap{S})$ satisfies {\bf (F)}, then so does its exact completion $(\overline{\ct{E}}, \overline{\smallmap{S}})$
(see \refprop{presofFullness} below).
\item[(c)] If $(\ct{E},\smallmap{S})$ satisfies {\bf (F)}, then so does the associated pair $(\Eff_{\ct{E}}[\pca{A}],\smallmap{S}[\pca{A}])$ defined by realisability (this willed be proved in \cite{bergmoerdijk07c}).
\item[(d)] If $(\ct{E},\smallmap{S})$ satisfies {\bf (F)}, then so does the associated pair $(\shct{\ct{E}}{\ct{C}}, \smallmap{S}[\ct{C}])$ defined by the sheaves (this willed be proved in \cite{bergmoerdijk07d}).
\end{itemize}
Of these, statement (a) is easy to prove, but the proofs of the other
three statements are non-trivial and technically rather involved, as
we will see.

This series of papers is not the first to make an attempt at satisfying
these goals. In particular, the authors of \cite{moerdijkpalmgren00} provided a suitable categorical treatment of inductive types, and used these in \cite{moerdijkpalmgren02} in an attempt to find a notion of ``predicative topos equipped with a class of small maps'' for which (i) and (ii) could be proved. The answer they gave, in terms of stratified pseudo-toposes,  was somewhat unsatisfactory in various ways: it used the categorical analogue of an infinite sequence of ``universes'', and involved a strengthening of {\bf CZF} by the axiom {\bf AMC} of ``multiple choice''. This was later improved upon by \cite{berg05a}, who established results along the lines of aim (ii) without using {\bf AMC}, but still involved universes. Awodey and Warren, in
\cite{awodeywarren05}, gave a much weaker axiomatisation of a ``predicative topos equipped with a class of small maps'', which didn't involve W-types, but for which they proved a completeness result along the lines of (iv). Gambino in \cite{gambino05} also proved a completeness theorem, and showed that unpublished work of Scott on presheaf models for set theory could be recovered in the context of Algebraic Set Theory. Later in \cite{gambino07}, he took a first step towards (ii) by showing the possibility of constructing the associated sheaf functor in a weak metatheory. In \cite{warren07}, Warren shows the stability of various axioms under coalgebras for a cartesian comonad.

To conclude this introduction, we will describe in more detail the
contents of this paper and its two sequels.

We begin this paper by making explicit the notion of ``category \ct{E} with a
class \smallmap{S} of small maps''.  Our axiomatisation, presented in Section 2, is based on various earlier such notions in the literature, in particular
the one in \cite{joyalmoerdijk95}, but is different from all of them. In particular, like the one in \cite{moerdijkpalmgren02}, our axiomatisation is meant to apply in the predicative context as well, but has a rather different flavour: unlike \cite{moerdijkpalmgren02}, we assume all diagonals
to be small, work with a weaker version of the representability
axiom, assume the Quotients axiom and work with Fullness instead of {\bf AMC}. In the same section 2, we will also introduce the somewhat weaker notion of a class \smallmap{S} of ``display maps'', and prove that any such class can be completed to a class $\smallmap{S}^{\rm cov}$ which satisfies all our axioms for small maps. In Section 3, we will consider various additional axioms which a class of small maps might satisfy. These additional requirements are all motivated by the axioms of set theories such as {\bf IZF} and {\bf CZF} (cf.~Appendix A for the axioms of {\bf IZF} and {\bf CZF}). Examples are the categorical Fullness axiom {\bf (F)} already mentioned above, and the axioms {\bf (WE)} and {\bf (WS)} which express that certain inductive W-types exist, respectively exist and are small. The core of the paper is formed by Sections 4--6, where we discuss exact completion. In Section 4, we will introduce a notion of exactness for categories with small maps $(\ct{E},\smallmap{S})$, essentially expressing that \ct{E} is closed under quotients by ``small'' equivalence relations. In Section 5, we use the familiar exact completion of regular categories \cite{menni00} to prove that any such pair $(\ct{E},\smallmap{S})$ possesses an exact completion $(\overline{\ct{E}}, \overline{\smallmap{S}})$. In Section  6, we then prove that the additional axioms for classes of small maps, such as Fullness and the existence of W-types, are preserved by exact completion. Some of the proofs in this section are quite
involved, and probably constitute the main new technical contribution
to Algebraic Set Theory contained in this paper. In Sections 7 and 8, we
return to the constructive set theories {\bf IZF} and {\bf CZF}, and show that our theory of exact pair $(\ct{E}, \smallmap{S)}$ of categories with small
maps provides a sound and complete semantics for these set theories. In
particular, in these two sections we achieve goals (i) and (iv) listed
above.

All the notions and results discussed in the present paper will be used
in the second and third papers in this series \cite{bergmoerdijk07c,bergmoerdijk07d}, where we will address realisability and sheaves. In the second paper, we will construct for any category with small maps $(\ct{E},\smallmap{S})$ a new category $\Asm_{\ct{E}}[\pca{A}]$ of assemblies equipped with a class of display maps $\smallmap{D}[\pca{A}]$. For this pair, we will show that its exact completion again satifies all our axioms for small maps. The model of set theory contained in this exact completion is a realisability model for constructive set theory {\bf CZF}, which coincides with the one by Rathjen in \cite{rathjen06}. We also plan to explain how a model construction by Streicher \cite{streicher05} and Lubarsky \cite{lubarsky06} fits into our framework.

The third paper will then address presheaf and sheaf models. First of all, we extend the work by Gambino in \cite{gambino05} to cover presheaf models for {\bf CZF}. Furthermore, for any category with small maps $(\ct{E}, \smallmap{S})$ and internal site \ct{C}, satisfying appropriate smallness conditions, we will define a class of small maps $\smallmap{S}[\ct{C}]$ in the category of internal sheaves in \ct{E}, resulting in a pair $(\shct{\ct{E}}{\ct{C}}, \smallmap{S}[\ct{C}])$. The validity of additional axioms for small maps is preserved through the construction, and, as a consequence, we obtain a theory of sheaf models for {\bf CZF} (extending the work in \cite{gambino06} on Heyting-valued models).

Acknowledgements: Throughout our work on the subject of this paper and its two sequels, we have been helped by discussions with many colleagues. In particular, we would like to mention Steve Awodey, Nicola Gambino, Per Martin-L\"of, Jaap van Oosten, Erik Palmgren, Michael Rathjen and Thomas Streicher. Last but not least, we would like to thank the editors for their patience.

\part*{The categorical setting}

The contents of this part of the paper are as follows. We first present the basic categorical framework for studying models of set theory in Section 2: a category with small maps. We give the axioms for a class of small maps, and also present the weaker notion of a class of display maps, and show how it generates a class of small maps. This will become relevant in our subsequent work on realisability. In Section 3 we will present additional axioms for a class of small maps, allowing use to model the set theories {\bf IZF} and {\bf CZF}. 

Throughout the entire paper, we will work in a positive Heyting category \ct{E}. For the definition of a positive Heyting category, and that of other categorical terminology, the reader is referred to Appendix B.

\section{Categories with small maps}

The categories we use to construct models of set theory we will call \emph{categories with small maps}. These are positive Heyting categories \ct{E} equipped with a class of maps \smallmap{S} satisfying certain axioms. The intuitive idea is that the objects in the positive Heyting category \ct{E} are \emph{classes}, and the maps \func{f}{B}{A} in \smallmap{S} are those class maps all whose fibres $B_a = f^{-1}(a)$ for $a \in A$ are ``small'', i.e., \emph{sets} in some (possibly rather weak) set theory . For this reason, we call the class \smallmap{S} a \emph{class of small maps}. So a map \func{f}{B}{A} belonging to such a class \smallmap{S} is an $A$-indexed family $(B_a)_{a \in A}$ of small subobjects of $B$.

\subsection{Classes of small maps}

We introduce the notion of a class of small maps.

\begin{defi}{localfullsubc}
A class of morphisms \smallmap{S} in a positive Heyting category \ct{E} will be called a \emph{locally full subcategory}, when it satisfies the following axioms:
\begin{description}
\item[(L1)] (Pullback stability) In any pullback square
\diag{ D \ar[d]_g \ar[r] & B \ar[d]^f \\
C \ar[r] & A }
where $f \in \smallmap{S}$, also $g \in \smallmap{S}$.
\item[(L2)] (Sums) Whenever $X \rTo Y$ and $X'\rTo Y'$ belong to \smallmap{S}, so does $X + X' \rTo Y + Y'$.
\item[(L3)] (Local Fullness) For a commuting triangle
\diag{ Z \ar[dr]_h \ar[rr]^f & & Y \ar[dl]^g \\
& X &  }
where $g \in \smallmap{S}$, one has $f \in \smallmap{S}$ iff $h \in \smallmap{S}$.
\end{description}
\end{defi}

When a locally full subcategory \smallmap{S} has been fixed together with an object $X \in \ct{E}$, we write $\smallmap{S}_X$ for the full subcategory of $\ct{E}/X$ whose objects are morphisms $A \rTo X \in \smallmap{S}$.

\begin{defi}{localHeytingcat}
A locally full subcategory \smallmap{S} will be called a \emph{locally full positive Heyting subcategory}, when every $\smallmap{S}_X$ is a positive Heyting category and the inclusion $\smallmap{S}_X \rTo \ct{E}/X$ preserves this structure.
\end{defi}

To complete the definition a class of small maps, we introduce the notion of a covering square.

\begin{defi}{coveringsquare}
A diagram
\diag{ A \ar[d]_f \ar[r] &  B \ar[d]^g \\
C \ar[r]_p & D}
is called a \emph{quasi-pullback}, when the canonical map $A \rTo B \times_D C$ is a cover. If $p$ is also a cover, the diagram will be called a \emph{covering square}. When $f$ and $g$ fit into a covering square as shown, we say that $f$ \emph{covers} $g$, or that $g$ \emph{is covered by} $f$.
\end{defi}

\begin{lemm}{covsqpasting} In a positive Heyting category \ct{E},
\begin{enumerate}
\item covering squares are stable under pullback. More explicitly, pulling back a covering square of the form
\diag{ A \ar[d] \ar@{->>}[r] &  B \ar[d] \\
C \ar@{->>}[r] & D}
along a map \func{p}{E}{D} results in a covering square of the form
\diag{ p^*A \ar[d] \ar@{->>}[r] &  p^*B \ar[d] \\
p^*C \ar@{->>}[r] & E.}
\item the juxtaposition of two covering squares as in the diagram below is again a covering square. 
\diag{ A \ar[d]_f \ar@{->>}[r] &  B \ar[d]_g \ar@{->>}[r] & C \ar[d]^h \\
X \ar@{->>}[r] & Y \ar@{->>}[r] & Z}
So, when $f$ covers $g$ and $g$ covers $h$, $f$ covers $h$. 
\item the sum of two covering squares is a covering square. More explicitly, when both
\begin{displaymath}
\begin{array}{ccc}
\xymatrix{ A_0 \ar[d]_{f_0} \ar@{->>}[r] &  B_0 \ar[d]^{g_0} \\
C_0 \ar@{->>}[r] & D_0}
& \mbox{and} &
\xymatrix{ A_1 \ar[d]_{f_1} \ar@{->>}[r] &  B_1 \ar[d]^{g_1} \\
C_1 \ar@{->>}[r] & D_1}
\end{array}
\end{displaymath}
are covering squares, then so is
\diag{ A_0 + A_1 \ar[d]_{f_0 + f_1} \ar@{->>}[r] &  B_0 + B_1 \ar[d]^{g_0 + g_1} \\
C_0 + C_1 \ar@{->>}[r] & D_0 + D_1.}
Therefore, if $f_0$ covers $g_0$ and $f_1$ covers $g_1$, then $f_0 + f_1$ covers $g_0 + g_1$.
\end{enumerate}
\end{lemm}
\begin{proof}
All straightforward consequences of the regularity of \ct{E}.
\end{proof}

\begin{defi}{classofsmallmaps}
A locally full positive Heyting subcategory \smallmap{S} is a \emph{class of small maps} when it satisfies the following two axioms:
\begin{description}
\item (Collection) Any two arrows \func{p}{Y}{X} and \func{f}{X}{A} where $p$ is a cover and $f$ belongs to \smallmap{S} fit into a covering square
\diag{ Z \ar[d]_g \ar[r] & Y \ar@{->>}[r]^p & X \ar[d]^f \\
B \ar@{->>}[rr]_h & & A}
where $g$ belongs to \smallmap{S}.
\item (Covered maps) When an arbitrary map $g$ is covered by a map $f \in \smallmap{S}$, also $g \in \smallmap{S}$.
\end{description}
\end{defi}

\begin{defi}{catwsmallmaps}
A pair $(\ct{E}, \smallmap{S})$, in which \ct{E} is a positive Heyting category and \smallmap{S} a class of small maps, will be called a \emph{category with small maps}. A \emph{morphism of categories with small maps} \func{F}{(\ct{E}, \smallmap{S})}{(\ct{F}, \smallmap{T})} is a functor $F$ that preserves the positive Heyting structure and sends maps in \smallmap{S} to maps in \smallmap{T}.
\end{defi}

\begin{rema}{infexample}
There is one informal example of a category with small maps that the reader should try to keep in mind. Let \ct{E} be the category of classes and let \smallmap{S} consist of those class morphisms all whose fibres are sets. The notions of class and set here can be understood in some intuitive sense, or can be made precise by a formal set theory like {\bf IZF} or {\bf CZF}. It is not too hard to see that this is indeed an example. We will flesh out this informal example in two different ways in Section 8.
\end{rema}

\begin{rema}{stabunderslicing}
An essential fact about categories with small maps is their stability under slicing. By this we mean that for any category with small maps $(\ct{E}, \smallmap{S})$ and object $X$ in \ct{E}, the pair $(\ct{E}/X, \smallmap{S}/X)$, with $\smallmap{S}/X$ being defined by
\[ f \in \smallmap{S}/X \Leftrightarrow \Sigma_X f \in \smallmap{S}, \] 
is again a category with small maps. The verification of this claim is straightforward and omitted. 

Strengthened versions of a category with small maps obtained by imposing more requirements on the class of small maps should also be stable under slicing in this sense. Therefore, when we introduce additional axioms for a class of small maps \smallmap{S} in a category \ct{E}, their validity should be inherited by the classes of small maps $\smallmap{S}/X$ in  $\ct{E}/X$. This will indeed be the case, but we will not point this out explicitly everytime we introduce an axiom, and a proof of its stability under slicing will typically be left to the reader.
\end{rema}

When a class of small maps \smallmap{S} in a positive Heyting category \ct{E} has been fixed, we refer to the morphisms in \smallmap{S} as the \emph{small maps}. Objects $X$ for which the unique map $X \rTo 1$ is small, will be called \emph{small}. Furthermore, a subobject $A \subseteq X$ represented by a monomorphism $A \rTo X$ belonging to $\smallmap{S}$ will be called \emph{bounded}. 

\begin{rema}{boundedsepforlocHeytcat} Throughout the paper, we will make use of the following internal form of ``bounded separation''. If $\phi(x)$ is a formula in the internal logic of \ct{E} with free variable $x \in X$, all whose basic predicates are bounded, and contains existential and universal quantifications $\exists_f$ and $\forall_f$ along small maps $f$ only, then
\[ A = \{ x \in X \, : \, \phi(x) \} \subseteq X \]
defines a bounded subobject of $X$. In particular, smallness of $X$ implies smallness of $A$. This is an immediate consequence of the fact that a class of small maps is a locally full positive Heyting subcategory.
\end{rema}

It will be convenient to also have a less comprehensive and more elementary axiomatisation of the notion of a class of small maps available, as provided by the next proposition. It will also facilitate the comparison with other definitions of a class of small maps to be found in the literature (cf.~\refrema{compwithothsettings} below).

\begin{prop}{altcharsmallmaps}
A class of maps \smallmap{S} in a positive Heyting category \ct{E} is a class of small maps iff it satisfies the following axioms:
\begin{description}
\item[(A1)] (Pullback stability) In any pullback square
\diag{ D \ar[d]_g \ar[r] & B \ar[d]^f \\
C \ar[r]_p & A }
where $f \in \smallmap{S}$, also $g \in \smallmap{S}$.
\item[(A2)] (Descent) If in a pullback square as above $p$ is a cover and $g \in \smallmap{S}$, then also $f \in \smallmap{S}$.
\item[(A3)] (Sums) Whenever $X \rTo Y$ and $X'\rTo Y'$ belong to \smallmap{S}, so does $X + X' \rTo Y + Y'$.
\item[(A4)] (Finiteness) The maps $0 \rTo 1, 1 \rTo 1$ and $1+1 \rTo 1$ belong to \smallmap{S}.
\item[(A5)] (Composition) $\smallmap{S}$ is closed under composition.
\item[(A6)] (Quotients) In a commuting triangle
\diag{ Z \ar[dr]_h \ar@{->>}[rr]^f & & Y \ar[dl]^g \\
& X, &  }
if $f$ is a cover and $h$ belongs to \smallmap{S}, then so does $g$.
\item[(A7)] (Collection) Any two arrows \func{p}{Y}{X} and \func{f}{X}{A} where $p$ is a cover and $f$ belongs to \smallmap{S} fit into a covering square
\diag{ Z \ar[d]_g \ar[r] & Y \ar@{->>}[r]^p & X \ar[d]^f \\
B \ar@{->>}[rr]_h & & A,}
where $g$ belongs to \smallmap{S}.
\item[(A8)] (Heyting) For any morphism \func{f}{Y}{X} belonging to \smallmap{S}, the right adjoint
\[ \func{\forall_f}{{\rm Sub}(Y)}{{\rm Sub}(X)} \]
sends bounded subobjects to bounded subobjects.
\item[(A9)] (Diagonals) All diagonals \func{\Delta_X}{X}{X \times X} belong to \smallmap{S}.
\end{description}
\end{prop}
\begin{proof}
Axioms {\bf (A1, 3, 5, 7, 9)} hold for any class of small maps by definition. Axioms {\bf (A2)} and {\bf (A6)} are equivalent to saying that \smallmap{S} is closed under covered maps. {\bf (A4)} holds because $\smallmap{S}_1$ is a lextensive category, and the inclusion in \ct{E} preserves this, while {\bf (A8)} holds because every $\smallmap{S}_X$ is Heyting, and the inclusion in $\ct{E}/X$ preserves this.

Conversely, let \smallmap{S} is a class of maps satisfying {\bf (A1-9)}. It will follow from the lemma below that \smallmap{S} is a locally full subcategory. Because \smallmap{S} satisfies Collection and is closed under covered maps by assumption, it remains to show that it is a locally full positive Heyting category. So let $X \in \ct{E}$ be arbitrary: $\smallmap{S}_X$ inherits the terminal object (by $1 \rTo 1 \in \smallmap{S}$ and pullback stability), pullbacks (by {\bf (A1)} and {\bf (A5)}) and the finite sums (by {\bf (A4)}, pullback stability and {\bf (A3)}) from $\ct{E}/X$. Finally, the regular structure it inherits by {\bf (A6)} and the Heyting structure by {\bf (A8)}.
\end{proof}
\begin{lemm}{composlemma}
Let \smallmap{S} be a class of maps satisfying the axioms {\bf (A1), (A5)} and {\bf (A9)}. If in a commuting triangle
\diag{ Z \ar[dr]_h \ar[rr]^f & & Y \ar[dl]^g \\
& X, &  }
$h$ belongs to \smallmap{S}, then so does $f$.
\end{lemm}
\begin{proof}
By the universal property of the pullback $Y \times_X Z$ we obtain a map $\rho = \langle f, \id \rangle$ making the diagram
\diag{Z \ar@/_/[ddr]_f \ar@/^/[drr]^{\id} \ar[dr]^{\rho} \\
& Y \times_X Z \ar[r]^{p_2} \ar[d]_{p_1} & Z \ar[d]^h \\
& Y \ar[r]_g & X}
commute. It suffices to show that $\rho$ belongs to \smallmap{S}, because $p_1$ belongs to \smallmap{S} by pullback stability and \smallmap{S} is closed under composition. But this follows by pullback stability as both squares in the diagram
\diag{Z \ar[d]_{\rho} \ar[rr]^f & & Y \ar[d]^{\Delta_g} \ar[r] & Y \ar[d]^{\Delta} \\
Y \times_X Z \ar[rr]_{\id \times_X f} & & Y \times_X Y \ar[r] & Y \times Y}
are readily seen to be pullbacks.
\end{proof}

\subsection{Classes of display maps}

In our subsequent work on realisability \cite{bergmoerdijk07c}, classes of small maps are obtained from something we will call \emph{classes of display maps}.

\begin{defi}{displaymaps}
A locally full Heyting subcategory \smallmap{S} will be called a \emph{class of display maps}, when it satisfies the Collection axiom {\bf (A7)} and the Diagonal axiom {\bf (A9)}.
\end{defi}
\begin{prop}{chardisplmaps}
A class of maps \smallmap{S} in a positive Heyting category \ct{E} is a class of display maps iff it satisfies the axioms {\bf (A1), (A3-5), (A7-9)}, and
\begin{description}
\item[(A10)] (Images) If in a commuting triangle
\diag{ Z \ar[dr]_f \ar@{->>}[rr]^e & & Y \ar@{ >->}[dl]^m \\
& X, &  }
$e$ is a cover, $m$ is monic, and $f$ belongs to \smallmap{S}, then also $m$ belongs to \smallmap{S}.
\end{description}
\end{prop}
\begin{proof}
As in \refprop{altcharsmallmaps}. Like for small maps, axioms {\bf (A1, 3, 5, 7, 9)} hold for any class of display maps by definition. Axiom {\bf (A4)} holds because $\smallmap{S}_1$ is a lextensive category, and the inclusion in \ct{E} preserves this, {\bf (A10)} holds because every $\smallmap{S}_X$ is regular, and the inclusion in $\ct{E}/X$ preserves this and {\bf (A8)} holds because every $\smallmap{S}_X$ is Heyting, and the inclusion in $\ct{E}/X$ preserves this.

Conversely, let \smallmap{S} be a class of maps satisfying {\bf (A1), (A3-5), (A7-10)}. As \smallmap{S} is a locally full subcategory by \reflemm{composlemma}, and satisfies Collection and contains all diagonals by assumption, all that has to be shown is that \smallmap{S} is a locally full positive Heyting category. But that follows in the manner we have seen, using {\bf (A10)} to show that all $\smallmap{S}_X$ are regular.
\end{proof}

The proposition we just proved explains that a class of display maps is like a class of small maps, except that it need not be closed under covered maps. More precisely, it need not satisfy the Descent axiom {\bf (A3)}, and it may satisfy the Quotients axiom {\bf (A6)} only in the weaker form of {\bf (A10)}. It should be pointed out that notions that we have defined for a class of small maps, like boundedness of subobjects, can also be defined for a class of display maps. And observe that \refrema{boundedsepforlocHeytcat} applies to classes of display maps as well.

The following proposition makes clear how a class of display maps generates a class of small maps.
\begin{prop}{displtosmallmaps}
Let \ct{E} be a category with a class of display maps \smallmap{S}. Then there is a smallest class of small maps $\smallmap{S}^{\rm cov}$ containing \smallmap{S}, where the maps that belong to $\smallmap{S}^{\rm cov}$ are precisely those that are covered by morphisms in \smallmap{S}.
\end{prop}

The proof relies on the following lemma, which makes use of the Collection axiom {\bf (A7)}.

\begin{lemm}{1stuseofcoll}
Any two maps \func{f}{Y}{X} and \func{g}{Z}{Y} belonging to $\smallmap{S}^{\rm cov}$ fit into a diagram of the form
\diag{ Z' \ar@{->>}[r] \ar[d]_{g'} & Z \ar[d]^g \\
Y' \ar@{->>}[r] \ar[d]_{f'} & Y \ar[d]^{f} \\
X' \ar@{->>}[r] & X, }
where both squares are covering squares and $g'$ and $f'$ belong to \smallmap{S}.
\end{lemm}
\begin{proof}
By definition of $\smallmap{S}^{\rm cov}$, $g$ and $f$ fit the diagram
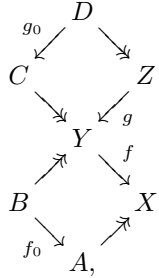
\begin{displaymath}
\xymatrix@!0{ & D \ar[dl]_{g_0} \ar@{->>}[dr] & \\
C \ar@{->>}[dr] & & Z \ar[dl]^g \\
& Y \ar[dr]^f & \\
B \ar@{->>}[ur] \ar[dr]_{f_0} & & X \\
& A, \ar@{->>}[ur] & }
\end{displaymath}
with $f_0, g_0 \in \smallmap{S}$ and the squares covering. We compute the pullback $B \times_Y C$, and then apply Collection to obtain a map $f' \in \smallmap{S}$ fitting into the diagram
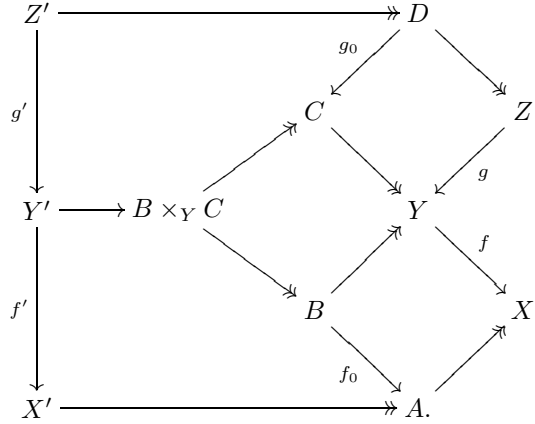
\diag{ Z' \ar[dd]_{g'} \ar@{->>}[rrr]  & & & D \ar[dl]_{g_0} \ar@{->>}[dr] & \\
 & & C \ar@{->>}[dr] & & Z \ar[dl]^g \\
Y' \ar[r] \ar[dd]_{f'} & B \times_Y C \ar@{->>}[dr] \ar@{->>}[ur] & & Y \ar[dr]^f & \\
&  & B \ar@{->>}[ur] \ar[dr]_{f_0} & & X \\
X' \ar@{->>}[rrr] &  & & A. \ar@{->>}[ur] & }
In this picture, the map $g'$ is obtained by pulling back $g_0$, so also this map belongs to \smallmap{S}. This finishes the proof.
\end{proof}

\begin{proof} (Of \refprop{displtosmallmaps}.)
The class of maps $\smallmap{S}^{\rm cov}$ is closed under covered maps by \reflemm{covsqpasting}, so {\bf (A2)} and {\bf (A6)} follow immediately. The validity of the axiom {\bf (A3)} for $\smallmap{S}^{\rm cov}$ follows from \reflemm{covsqpasting} as well. Validity of {\bf (A4)} and {\bf (A9)} follows simply because $\smallmap{S} \subseteq \smallmap{S}^{\rm cov}$, while that of {\bf (A5)} follows from the previous lemma. The other axioms present more difficulties.

{\bf (A1)}: Assume $f$ can be obtained by pullback from a map $g \in \smallmap{S}^{\rm cov}$. We will construct a cube involving $f$ and $g$ of the form
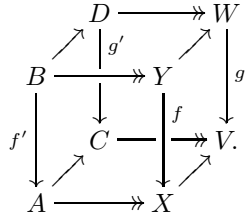
\begin{displaymath}
\xymatrix@!0{
 & D  \ar@{->>}[rr] \ar'[d]^{g'}[dd] & & W \ar[dd]^g \\
B \ar@{->>}[rr] \ar[dd]_{f'} \ar[ur] &  & Y \ar[dd]^(.3)f \ar[ur]  \\
 & C \ar@{->>}'[r][rr] & & V. \\
A \ar@{->>}[rr] \ar[ur] &  & X \ar[ur] }
\end{displaymath}
We begin by choosing a covering square at the back with $g' \in \smallmap{S}$. Next, the front is obtained by pulling back the square at the back along the map $X \rTo V$. This makes the front a covering square as well (by \reflemm{covsqpasting}), and all the other faces pullbacks. Therefore $f' \in \smallmap{S}$,  by pullback stability of \smallmap{S}, so that $f \in \smallmap{S}^{\rm cov}$.

{\bf (A7)}: Let $\func{f}{Y}{X} \in \smallmap{S}^{\rm cov}$ and a cover $Z \rTo Y$ be given. We obtain a diagram 
\begin{displaymath}
\xymatrix@!0{
  & & & Z  \ar@{->>}[rr]  & & Y \ar[dd]^f \\
D \ar[dd]_{g'} \ar[rr]  & & P \ar[ur]\ar@{->>}[rr] & & B \ar[dd]^(.4){f'} \ar[ur]  \\
 & & & & & X. \\
C \ar@{->>}[rrrr]   & &  & & A \ar[ur] }
\end{displaymath}
The map $f' \in \smallmap{S}$ covering $f$ exists by definition of $\smallmap{S}^{\rm cov}$. Next, we apply Collection to $f'$ and the cover $P \rTo B$ obtained by pullback. This results in a map $g' \in \smallmap{S}$ covering $f'$, and hence also $f$.

{\bf (A8)}: Let $\func{f}{Y}{X}$ be a map belonging to $\smallmap{S}^{\rm cov}$, and let $A$ be an $\smallmap{S}^{\rm cov}$-bounded subobject of $Y$. Using the previous lemma, we obtain a diagram 
\diag{  A' \ar@{->>}[r] \ar[d]_{i'} & A \ar@{ >->}[d]^i \\
 Y' \ar@{->>}[r]^q \ar[d]_{ f'} & Y \ar[d]^{f} \\
 X' \ar@{->>}[r]_p & X, }
with $i', f' \in \smallmap{S}$ and both squares covering. We may actually assume that the top square is a pullback and $i'$ is monic (replace $i'$ by its image and use {\bf (A10)} if necessary). We can now use the following formula for $\forall_f (i)$ to see that it is $\smallmap{S}^{\rm cov}$-bounded:
\[ \forall_f (i) = \exists_p \forall_{f'} (i'). \]
For $\forall_{f'} (i')$ is an \smallmap{S}-bounded subobject of $X'$, since {\bf (A8)} holds for \smallmap{S}, and hence $\exists_p \forall_{f'} (i')$ is a $\smallmap{S}^{\rm cov}$-bounded subobject of $X$ by the Descent axiom {\bf (A2)} for $\smallmap{S}^{\rm cov}$.
\end{proof}

\begin{rema}{origindisplay}
A result closely related to \refprop{displtosmallmaps} can already be found in \cite{joyalmoerdijk94}. We have borrowed the term ``display map'' from sources such as \cite{hylandpitts89}, where classes of maps with similar properties were used to provide a categorical semantics for type theory.
\end{rema}

Like for small maps, a pair $(\ct{E}, \smallmap{S})$, where \ct{E} is a positive Heyting category and \smallmap{S} is a class of display maps, will be called \emph{a category with display maps}.

What does not seem to be true in general is that additional axioms on \smallmap{S}, such as those explained in the next section, are automatically inherited by $\smallmap{S}^{\rm cov}$. The question which additional properties are inherited is explored in Section 6, and it will be seen that the answer may depend on the exactness properties of \ct{E}.

\section{Axioms for classes of small maps}

For the purpose of modelling the set theories {\bf IZF} and {\bf CZF}, our notion of a category with small maps is too weak (the reader will find the axioms for these set theories in Appendix A below). Therefore we consider in this section various possible strengthenings, obtained by imposing further requirements on the class of small maps.  

For later use it is important to observe that the axioms make sense for a class of display maps as well. For this reason, our standing assumption throughout this section is that $(\ct{E}, \smallmap{S})$ is a category with display maps.

\subsection{Representability}

\begin{defi}{representable}
A \emph{representation} for a class of display maps \smallmap{S} is a morphism $\func{\pi}{E}{U} \in \smallmap{S}$ such that any morphism $f \in S$ is covered by a pullback of $\pi$. More explicitly: any $\func{f}{Y}{X} \in \smallmap{S}$ fits into a diagram of the form
\diag{Y \ar[d]_f & A \ar[d] \ar[r] \ar@{->>}[l] & E \ar[d]^{\pi} \\
X & B \ar[r] \ar@{->>}[l] & U, }
where the left hand square is covering and the right hand square is a pullback. The class \smallmap{S} will be called \emph{representable}, if it has a representation.
\end{defi}
\begin{rema}{onrepr}
In \cite{joyalmoerdijk95}, the authors take as basic a different notion of representability. Even when these notions can be shown to be equivalent (as in \refprop{reprinJMsense}), it is the above notion we find easier to work with.
\end{rema}

\subsection{Separation}
 
For the purpose of modelling the Full separation axiom of {\bf IZF}, one may impose the following axiom:
\begin{description}
\item[(M)] All monomorphisms belong to \smallmap{S}.
\end{description}

\subsection{Power types}

Before we introduce an axiom corresponding to the Power set axiom of {\bf IZF}, we first formulate an axiom which imposes the existence of a power class object. Intuitively, the elements of the power class $\spower X$ of a class $X$ are the subsets of the class $X$. For our purposes it is important to realise that an axiom requiring the existence of a power \emph{class} is rather weak: it holds in \emph{every} set theory, even predicative ones like {\bf CZF}, and it is therefore not to be confused with the Power set axiom. 


\begin{defi}{powerobj}
By a \emph{$D$-indexed family of subobjects} of $C$, we mean a subobject $R \subseteq C \times D$. A $D$-indexed family of subobjects $R \subseteq C \times D$ will be called \emph{\smallmap{S}-displayed} (or simply \emph{displayed}), whenever the composite
\[ R \subseteq C \times D \rTo D \]
belongs to \smallmap{S}. If it exists, the \emph{power class object} $\spower X$ is the classifying object for the displayed families of subobjects of $X$. This means that it comes equipped with a displayed $\spower X$-indexed family of subobjects of $X$, denoted by $\in_X \subseteq X \times \spower X$ (or simply $\in$, whenever $X$ is understood), with the property that for any displayed $Y$-indexed family of subobjects of $X$, $R \subseteq X \times Y$ say, there exists a unique map \func{\rho}{Y}{\spower X} such that the square
\diag{ R \ar@{ >->}[d] \ar[r] & \in_X \ar@{ >->}[d] \\
X \times Y \ar[r]_{\id \times \rho} & X \times \spower X}
is a pullback.
\end{defi}

This leads to the following axiom for a class of display maps \smallmap{S}:
\begin{description}
\item[(PE)] For any object $X$ the power class object $\spower X$ exists.
\end{description}
For once, we will briefly indicate why this axiom is stable under slicing:
\begin{lemm}{PEstableunderslicing}
If $(\ct{E}, \smallmap{S})$ is a category with of a class of display maps satisfying {\bf (PE)} and $X$ is any object in \ct{E}, then $\smallmap{S}/X$ also satisfies {\bf (PE)} in $\ct{E}/X$. Moreover, $\spower$ is an indexed endofunctor.
\end{lemm}
\begin{proof}
If \func{f}{Y}{X} is an object of $\ct{E}/X$, then $\slspower{X}(f) \rTo X$ is
\[ \{ (x \in X, \alpha \in \spower(Y)) \, : \, \forall y \in \alpha \,  f(y) = x \}, \]
together with the projection on the first component.
\end{proof}

As discussed already in \cite{joyalmoerdijk95}, the assignment $X \mapsto \spower X$ is functorial for a class of small maps for which {\bf (PE)} holds (we doubt whether the same is true for a class of display maps). In fact, in this case $\spower$ is the functor part of a monad, with a unit \func{\eta_X}{X}{\spower X} and a multiplication \func{\mu_X}{\spower \spower X}{\spower X} which can be understood intuitively as singleton and union. We refer to \cite{joyalmoerdijk95} for a discussion of these points. We also borrow from \cite{joyalmoerdijk95} the following proposition, which we will have to invoke later.
\begin{prop}{collandPE} \cite[Proposition I.3.7]{joyalmoerdijk95}
When \smallmap{S} is a class of small maps satisfying {\bf (PE)}, then $\spower$ preserves covers.
\end{prop}


\begin{rema}{boundedsubobjclass} For a class of small maps \smallmap{S}, the object $\Omega_b = \spower 1$ could be called the object of bounded truth-values, or the bounded subobject classifier, as the subobject $\in$ of  $1 \times \spower 1 \cong \spower 1$ classifies bounded subobjects: for any mono $\func{m}{A}{X}$ in $\smallmap{S}$ there is a unique map $\func{c_m}{X}{\spower 1}$ such that
\diag{ A \ar@{ >->}[d]_m \ar[r] & \in \ar@{ >->}[d] \\
X \ar[r]_{c_m} & \spower 1 }
is a pullback. Actually, as for the ordinary subobject classifier in a topos, it can be shown that the domain of the map $\in \rTo \spower 1$ is isomorphic to the terminal object 1. Moreover, internally, $\spower 1$ has the structure of a poset with small infima and suprema, implication, and top and bottom. This is a consequence of the fact that the maximal and minimal subobject are bounded, and bounded subobjects are closed under implication, union, intersection, existential and universal quantification. Another way of expressing this would be to say that bounded truth-values are closed under truth and falsity, implication, conjunction and disjunction, and existential and universal quantification over small sets. The classifying bounded mono $1 \rTo \spower 1$ will therefore be written $\top$ (for ``true'' or ``top''), as it points to the top element of the poset $\spower 1$.

A formula $\phi$ in the internal language will be said to have a bounded truth-value, when
\[ \exists p \in \spower 1 \, ( \, \phi \leftrightarrow p = \top \, ), \]
or, equivalently,
\[ \exists p \in \spower 1 \, ( \, \phi \leftrightarrow * \in p \, ), \]
if $*$ is the unique element of $1$. Notice that in both cases a $p \in \spower 1$ having the required property is automatically unique. Note also that for a subobject $A \subseteq X$, saying that $x \in A$ has a bounded truth-value for all $x \in X$ is the same as saying that $A$ is a bounded subobject of $X$.
\end{rema}

For a class of display maps \smallmap{S} satisfying {\bf (PE)} we can now state the axiom we need to model the Power set axiom of {\bf IZF}.
\begin{description}
\item[(PS)] For any map $\func{f}{Y}{X} \in \smallmap{S}$, the power class object $\slspower{X} (f) \rTo X$ in $\ct{E}/X$ belongs to \smallmap{S}.
\end{description}

\subsection{Function types}

We will now introduce the axiom {\bf ($\Pi$S)} reminiscent of the Exponentiation axiom in set theory. Before we do so, we first note an important consequence of the axiom {\bf (PE)}. 

Call a map \func{f}{Y}{X} in \ct{E} \emph{exponentiable}, if the functor \func{(-) \times f}{\ct{C}/X}{\ct{C}/X} has a right adjoint $(-)^f$, or, equivalently, if the functor \func{f^*}{\ct{C}/X}{\ct{C}/Y} has a right adjoint $\Pi_f$.
\begin{lemm}{PEimpliesPiE} \cite{awodeywarren05}
When a class of display maps satisfies {\bf (PE)}, then all display maps are exponentiable.
\end{lemm}
\begin{proof}
Since the axiom {\bf (PE)} is stable under slicing, it suffices to show that the object $X^A$ exists, when $A$ is small. But this can be constructed as:
\[ X^A := \{ \alpha \in \spower (A \times X) \, : \, \forall a \in A \, \exists ! x \in X \, (a, x) \in \alpha \}. \]
The required verifications are left to the reader.
\end{proof}

In certain circumstances, the converse holds as well (see \refcoro{PiEimpliesPE}).

One can formulate the conclusion of the preceding lemma as an axiom:
\begin{description}
\item[($\Pi$E)] All morphisms $f \in \smallmap{S}$ are exponentiable.
\end{description}
This axiom should not be associated with the Exponentiation axiom in set theory, which is more closely related to its strengthening {\bf ($\Pi$S)} below.
\begin{description}
\item[($\Pi$S)] For any map $\func{f}{Y}{X} \in \smallmap{S}$, the functor
\[ \func{\Pi_f}{\ct{E}/Y}{\ct{E}/X} \]
exists and preserves morphisms in \smallmap{S}.
\end{description}
Note that:
\begin{lemm}{PSimpliesPiS}
For a class of display maps \smallmap{S}, {\bf (PS)} implies {\bf ($\Pi$S)}.
\end{lemm}
\begin{proof}
As in \reflemm{PEimpliesPiE}.
\end{proof}

The converse is certainly false: the Exponentiation axiom is a consequence of {\bf CZF}, but the Power set axiom is not. (For a countermodel, see \cite{streicher05} and \cite{lubarsky06}. We will study this model further in the second paper of this series.)

\subsection{Inductive types}

In this section we want to discuss axioms concerning the existence and smallness of certain inductively defined structures. Our paradigmatic example of an inductively defined object is the W-type in Martin-L\"of's type theory \cite{martinlof84}. We will not give a review of the theory of W-types, but we do wish to give a complete explanation of how they are modelled categorically, following \cite{moerdijkpalmgren00}.

W-types are examples of initial algebras, and as we will meet other initial algebras as well, we will give the general definition.

\begin{defi}{initalg}
Let \func{T}{\ct{C}}{\ct{C}} be an endofunctor on a category \ct{C}. The category $\alg{T}$ of \emph{$T$-algebras} has as objects pairs $(A, \func{\alpha}{TA}{A})$, and as morphisms $(A, \alpha) \rTo (B, \beta)$ arrows \func{m}{A}{B} making the diagram
\diag{ TA \ar[r]^{Tm} \ar[d]_{\alpha} & TB \ar[d]^{\beta} \\
A \ar[r]_m & B}
commute. The initial object in this category (whenever it exists) is called the \emph{initial $T$-algebra}. In case $T$ is indexed endofunctor, the category of $\alg{T}$ of $T$-algebras is an indexed category, and the initial $T$-algebra will be called the \emph{indexed initial $T$-algebra} if all its reindexings are also initial in the appropriate fibres.
\end{defi}

An essential fact about initial algebras is that they are fixed points. A \emph{fixed point} for an endofunctor $T$ is an object $A$ together with an isomorphism $TA \cong A$. A lemma by Lambek \cite{lambek70} tells us that the structure map $\alpha$ of the initial algebra, assuming it exists, is an isomorphism, so that initial algebras are fixed points.

Another property of initial algebras is that they have no proper subalgebras: \func{m}{(A, \alpha)} {(B, \beta)} is a subalgebra of $(B, \beta)$, when $m$ is a monomorphism in \ct{C}. The subalgebra is called proper, in case $m$ is not an isomorphism in \ct{C}. That initial algebras have no proper subalgebras is usually related to an induction principle that they satisfy, while their initiality expresses that they allow definitions by recursion.

When a map \func{f}{B}{A} is exponentiable in a cartesian category \ct{E}, it induces an endofunctor on \ct{C}, which will be called the \emph{polynomial functor} $P_f$ associated to $f$. The quickest way to define it is as the following composition:
\diag{ \ct{C} \cong \ct{C}/1 \ar[r]^{B^*} & \ct{C}/B \ar[r]^{\Pi_f} & \ct{C}/A \ar[r]^{\Sigma_A} & \ct{C}/1 \cong \ct{C}. }
In more set-theoretic terms it could be defined as:
\[ P_f(X) = \sum_{a \in A} X^{B_a}. \]
Whenever it exists, the initial algebra for the polynomial functor $P_f$ will be called the W-type associated to $f$.

Intuitively, elements of a W-type are well-founded trees. In the category of sets, all W-types exist, and the W-types have as elements well-founded trees, with an appropriate labelling of its edges and nodes. What is an appropriate labelling is determined by the branching type \func{f}{B}{A}: nodes should be labelled by elements $a \in A$, edges by elements $b \in B$, in such a way that the edges into a node labelled by $a$ are enumerated by $f^{-1}(a)$. The following picture hopefully conveys the idea:
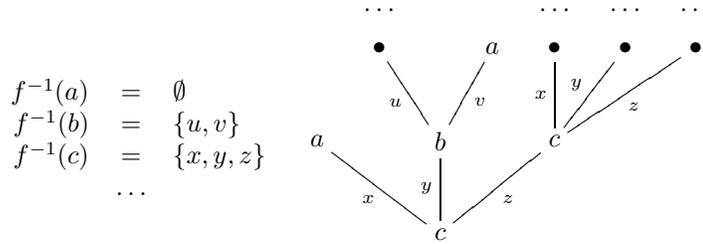
\begin{displaymath}
\xymatrix@C=.75pc@R=.5pc{ & & \ldots & & & \ldots & \ldots & \ldots \\
           & & {\bullet} \ar@{-}[dr]_u & & a \ar@{-}[dl]^v & {\bullet} \ar@{-}[d]_x & {\bullet} \ar@{-}[dl]_y & {\bullet} \ar@{-}[dll]^z \\
*{\begin{array}{rcl}
f^{-1}(a) & = & \emptyset \\
f^{-1}(b) & = & \{ u, v \} \\
f^{-1}(c) & = & \{ x, y, z \} \\
& \ldots & 
\end{array}}       & a \ar@{-}[drr]_x & & b \ar@{-}[d]_y & & c \ar@{-}[dll]^z & & \\
           & & & c & & & & } 
\end{displaymath}
This set has the structure of a $P_f$-algebra: when an element $a \in A$ is given, together with a map \func{t}{B_a}{W_f}, one can build a new element ${\rm sup}_a t \in W_f$, as follows. First take a fresh node, label it by $a$ and draw edges into this node, one for every $b \in B_a$, labelling them accordingly. Then on the edge labelled by $b \in B_a$, stick the tree $tb$. Clearly, this sup operation is a bijective map. Moreover, since every tree in the W-type is well-founded, it can be thought of as having been generated by a possibly transfinite number of iterations of this sup operation. That is precisely what makes this algebra initial. The trees that can be thought of as having been used in the generation of a certain element $w \in W_f$ are called its subtrees. One could call the trees $tb \in W_f$ the \emph{immediate subtrees} of ${\rm sup}_a t$, and $w' \in W_f$ a \emph{subtree} of $w \in W_f$ if it is an immediate subtree, or an immediate subtree of an immediate subtree, or\ldots, etc. Note that with this use of the word subtree, a tree is never a subtree of itself (so proper subtree might have been a better terminology).

This concludes our introduction to W-types.

In the presence of a class of display maps \smallmap{S} satisfying {\bf ($\Pi$E)}, we will consider the following two axioms for W-types:
\begin{description}
\item[(WE)] For all $\func{f}{X}{Y} \in \smallmap{S}$, $f$ has an indexed W-type $W_f$.
\item[(WS)] Moreover, if $Y$ is small, also $W_f$ is small.
\end{description}

\subsection{Infinity}

The following two axioms, which make sense for any class of display maps \smallmap{S}, are needed to model the Infinity axiom in {\bf IZF} and {\bf CZF}:
\begin{description}
\item[(NE)] \ct{E} has a natural numbers object $\NN$.
\item[(NS)] Moreover, $\NN \rTo 1 \in \smallmap{S}$.
\end{description}
In fact, this is a special case of the previous example, for the natural numbers object is the W-type associated to the left sum inclusion \func{i}{1}{1 + 1} (which is always exponentiable). So {\bf (WE)} implies {\bf (NE)} and {\bf (WS)} implies {\bf (NS)}.

\subsection{Fullness}

We have almost completed our tour of the different axioms for a class of small maps we want to consider. There is one axiom that is left, the Fullness axiom, which allows us to model the Subset collection axiom of {\bf CZF}. It should be considered as a strengthened version of the axiom {\bf ($\Pi$S)}.

Over the other axioms of {\bf CZF} the Subset collection axiom is equivalent to an axiom called Fullness (see \cite{aczelrathjen01}):
\begin{description}
\item[Fullness:] $\exists z\;\! (z \;\!\subseteq \;\!{\bf mvf}(a,b) \land \forall x \;\!\epsilon \;\!{\bf mvf}(a,b) \;\! \exists c \;\!\epsilon \;\!z \;\! (c \;\!\subseteq \;\!x))$,
\end{description}
where we have used the abbreviation ${\bf mvf}(a, b)$ for the class of multi-valued functions from $a$ to $b$, i.e., sets $r \subseteq a \times b$ such that $\forall x \epsilon a \, \exists y \epsilon b \, (x, y) \epsilon r$. In words, this axiom states that for any pair of sets $a$ and $b$, there is a set of multi-valued functions from $a$ to $b$ such than any multi-valued function from $a$ to $b$ contains one in this set. We find it more convenient to consider a slight reformulation of Fullness, which concerns multi-valued \emph{sections}, rather than multi-valued \emph{functions}. A multi-valued section (or \emph{mvs}) of a function \func{\phi}{b}{a} is a multi-valued function $s$ from $a$ to $b$ such that $\phi s = \id_a$ (as relations). Identifying $s$ with its image, this is the same as a subset $p$ of $b$ such that $p \subseteq b \rTo a$ is surjective. Our reformulation of Fullness states that for any such $\phi$ there is a small family of small \emph{mvs}s such that any \emph{mvs} contains one in this family. Written out formally:
\begin{description}
\item[Fullness':] $\exists z\;\! (z \;\!\subseteq \;\!{\bf mvs}(f) \land \forall x \;\!\epsilon \;\!{\bf mvs}(f) \;\! \exists c \;\!\epsilon \;\!z \;\! (c \;\!\subseteq \;\!x))$.
\end{description}
Here, ${\bf mvs}(f)$ is an abbreviation for the class of all multi-valued sections of a function \func{f}{b}{a}, i.e., subsets $p$ of $b$ such that $\forall x \epsilon a \, \exists y \epsilon p \, f(y) = x$. The two formulations of Fullness are clearly equivalent. (Proof: observe that multi-valued sections of $\phi$ are multi-valued functions from $a$ to $b$ with a particular $\Delta_0$-definable property, and multi-valued functions from $a$ to $b$ coincide with the multi-valued sections of the projection $a \times b \rTo a$.)

We now translate our formulation of Fullness in categorical terms. A multi-valued section (\emph{mvs}) for a map \func{\phi}{B}{A}, over some object $X$, is a subobject $P \subseteq B$ such that the composite $P \rTo A$ is a cover. We write
\[ {\rm mvs}_X(\phi) \]
for the set of all \emph{mvs}s of a map $\phi$. This set obviously inherits the structure of a partial order from Sub($B$).

Multi-valued sections have a number of stability properties. First of all, any morphism \func{f}{Y}{X} induces an order-preserving map
\[ {\rm mvs}_X(\phi) \rTo {\rm mvs}_Y (f^*\phi), \]
obtained by pulling back along $f$. To avoid overburdening the notation, we will frequently talk about the map $\phi$ over $Y$, when we actually mean the map $f^* \phi$ over $Y$, the map $f$ always being understood. 

Furthermore, in a covering square
\diag{ B_0 \ar[d]_{\phi_0} \ar[r]^{\beta} & B \ar[d]^{\phi} \\
A_0 \ar[r]_{\alpha} & A, }
the sets mvs($\phi_0$) and mvs($\phi$) are connected by a pair of adjoint functors. The right adjoint \func{\beta^*}{{\rm mvs}(\phi)}{{\rm mvs}(\phi_0)} is given by pulling back along $\beta$, and the left adjoint $\beta_*$ by taking the image along $\beta$. 

Suppose we have fixed a class of display maps \smallmap{S}. We will call a \emph{mvs} $P \subseteq B$ of \func{\phi}{B}{A} \emph{displayed}, when the composite $P \rTo A$ belongs to \smallmap{S}. In case $\phi$ belongs to \smallmap{S}, this is equivalent to saying that $P$ is a bounded subobject of $B$. 

If we assume that in a covering square as above $\phi$ and $\phi_0$ belong to \smallmap{S}, the pullback functor $\beta^*$ will map displayed \emph{mvs}s to displayed \emph{mvs}s. If we assume moreover that $\beta$, or $\alpha$, belongs to \smallmap{S}, also $\beta_*$ will preserve displayed \emph{mvs}s.

We can now state a categorical version of the Fullness axiom:\footnote{A version in terms of multi-valued functions was contained in \cite{bergdemarchi06}.} 
\begin{description}
\item[(F)] For any $\func{\phi}{B}{A} \in \smallmap{S}$ over some $X$ with $A \rTo X \in \smallmap{S}$, there is a cover \func{q}{X'}{X} and a map \func{y}{Y}{X'} belonging to $\smallmap{S}$, together with a displayed \emph{mvs} $P$ of $\phi$ over $Y$, with the following ``generic'' property: if \func{z}{Z}{X'} is any map and $Q$ any displayed \emph{mvs} of $\phi$ over $Z$, then there is a map \func{k}{U}{Y} and a cover \func{l}{U}{Z} with $yk = zl$, such that $k^* P \leq l^* Q$ as (displayed) \emph{mvs}s of $\phi$ over $U$.
\end{description}

\begin{rema}{FvsPiS}
For classes of small maps satisfying {\bf (PE)}, the axiom {\bf (F)} implies {\bf ($\Pi$S)}. For showing this implication for classes of display maps not necessarily satisfying {\bf (PE)}, some form of exactness seems to be required.
\end{rema}

\part*{Exact completion}

We now come to the technical heart of the paper. We present a further strengthening of the notion of a category with small maps in the form of exactness. In Section 4 we will argue both that it is a very desirable property for a category with small maps to have, and that we cannot expect every category with small maps to be exact. This motivates our work in Sections 5 and 6, where we show how every category with small maps can ``conservatively'' be embedded in an exact one. In Section 5 we show this for the basic structure, and in Section 6 for the extensions based on the presence of additional axioms for a class of small maps.

\section{Exactness and its applications}

Let us first recall the notion of exactness for ordinary categories.
\begin{defi}{exactness}
A subobject \diag{R \ar@{ >->}[r]^i & X \times X} in a cartesian category \ct{C} is called an \emph{equivalence relation} when for any object $A$ in \ct{C} the image of the injective function
\[ \mbox{Hom}(A,R) \rTo \mbox{Hom}(A, X \times X) \rTo \mbox{Hom}(A, X)^2 \]
is an equivalence relation on the set $\mbox{Hom}(A,X)$. In the presence of a class of small maps \smallmap{S}, the equivalence relation is called \emph{\smallmap{S}-bounded}, when $R$ is a \smallmap{S}-bounded subobject of $X \times X$.

A diagram of the form
\diag{
  A \ar@<.8ex>[r]^{r_0} \ar@<-.8ex>[r]_-{r_1} & B \ar[r]^q & Q
}
is called \emph{exact}, when it is both a pullback and coequaliser. The diagram is called \emph{stably exact}, when for any \func{p}{P}{Q} the diagram
\diag{
  p^{*}A \ar@<.8ex>[r]^{p^{*}r_0} \ar@<-.8ex>[r]_-{p^{*}r_1} & p^{*}B \ar[r]^{p^{*}q} & P
}
obtained by pullback is also exact. A morphism \func{q}{X}{Q} is called the \emph{(stable) quotient} of an equivalence relation \func{i}{R}{X \times X}, if the diagram
\diag{
  R \ar@<.8ex>[r]^{\pi_0i} \ar@<-.8ex>[r]_-{\pi_1i} & X \ar[r]^q & Q
}
is stably exact.

A cartesian category \ct{C} is called \emph{exact}, when every equivalence relation in \ct{C} has a quotient. A positive exact category is called a \emph{pretopos}, and a positive exact Heyting category a \emph{Heyting pretopos}.
\end{defi}

This notion of exactness is too strong for our purposes, in view of the following argument. Let \func{i}{R}{X \times X} be an equivalence relation that has a quotient \func{q}{X}{Q} in a category with small maps $(\ct{E}, \smallmap{S})$. Since diagonals belongs to \smallmap{S} and the following square is a pullback:
\diag{ R \ar@{ >->}[d]_i \ar[r] & Q \ar@{ >->}[d]^{\Delta_Q} \\
X \times X \ar[r]_{q \times q} & Q \times Q,}
$i$ belongs to \smallmap{S} by pullback stability. So all equivalence relations that have a quotient are bounded. So if one demands exactness, all equivalence relations will be bounded. The only case we see in which one can justify this consequence is in the situation were all subobjects are bounded (i.e., {\bf (M)} holds). But imposing such impredicative conditions on categories with small maps is inappropriate when studying predicative set theories like {\bf CZF}.

Two possibilities suggest themselves. One alternative would be to require the existence of quotients of \emph{bounded} equivalence relations only (the above argument makes clear that this is the maximum amount of exactness that can be demanded). The other possibility would be to drop the axiom {\bf (A9)} for a class of small maps, which requires the diagonals to be small. 

We find the first option preferable both technically and psychologically. Since objects that do not have a small diagonal play no r\^ole in the theory, it is more convenient to not have them around. Moreover, a number of our proofs depend on the fact that all diagonals are small: in particular, those of \reflemm{composlemma} and the results which make use of this lemma, and \refprop{redlemmaWtypes}. It is not clear to us if corresponding proofs can be found if not all diagonals are small. We also expect additional technical complications in the theory of sheaves when it is pursued along the lines of the second alternative. Finally, note that the ideal models in \cite{awodeyetal04} and \cite{awodeywarren05} only satisfy bounded exactness. Hence the following definition.

\begin{defi}{boundedex}
A category with small maps $(\ct{E}, \smallmap{S})$ will be called \emph{(bounded) exact}, when every $\smallmap{S}$-bounded equivalence relation has a quotient. 
\end{defi}

\begin{rema}{presofexdiagr}
Observe that a morphism \func{F}{(\ct{E}, \smallmap{S})}{(\ct{F}, \smallmap{T})} between categories with small maps, as a regular functor, will always map quotients of $\smallmap{S}$-bounded equivalence relations to quotients of $\smallmap{T}$-bounded equivalence relations.
\end{rema}

Exactness of a category with small maps has two important consequences. First of all, we can use exactness to prove that every category with a representable class of small maps satisfying the axioms {\bf ($\Pi$E)} and {\bf (WE)} contains a model of set theory. This will be \reftheo{existinitpsalg} below. 

The other important consequence, which we can only state but not explain in detail, is the existence of a sheafification functor. This is essential for developing a good theory of sheaf models in the context of Algebraic Set Theory. As is well-known (see e.g.~\cite{maclanemoerdijk92}), the sheafification functor is constructed by iterating the plus construction twice. But the plus construction is an example of a quotient construction: it builds the collection of all compatible families and then identifies those that agree on a common refinement. For this to work, some exactness is necessary. We will come back to this in subsequent work.

Another issue where exactness plays a role is the following. We have shown that any class of display maps \smallmap{S} generates a class of small maps $\smallmap{S}^{\rm cov}$ (see Section 2.2). As it turns out, showing that additional properties of \smallmap{S} are inherited by $\smallmap{S}^{\rm cov}$ sometimes seems to require the exactness of the underlying category, as will be discussed in Section 6 below.

As another application of exactness we could mention the following:
\begin{prop}{reprinJMsense}
Let $(\ct{E}, \smallmap{S})$ is an exact category with a representable class of small maps satisfying {\bf ($\Pi$E)}. Then there exists a ``universal small map'' in the sense of \cite{joyalmoerdijk95}, i.e., a representation \func{\pi'}{E'}{U'} for \smallmap{S} such that any $\func{f}{Y}{X}$ in \smallmap{S} fits into a diagram of the form
\diag{Y \ar[d]_f & A \ar[d] \ar[r] \ar[l] & E' \ar[d]^{\pi'} \\
X & B \ar[r] \ar@{->>}[l]^p & U', }
where both squares are pullbacks and $p$ is a cover.
\end{prop}
\begin{proof}
$U'$ will be constructed as:
\begin{eqnarray*}
 U' & =  & \{ (u \in U, v \in U, \func{p}{E_v}{E_u \times E_u}) \, :  \\
 & & {\rm Im}(p) \mbox{ is an equivalence relation on } E_u \},
\end{eqnarray*}
while the fibre of $E'$ above $(u, v, p)$ will be $E_u / {\rm Im}(p)$. To indicate briefly why this works: any small object $X$ is covered by some fibre $E_u$ via a cover \func{q}{E_u}{X}. The kernel pair of $q$ is an equivalence relation $R \subseteq E_u \times E_u$, which is bounded, since the diagonal $X \rTo X \times X$ is small. This means that $R$ is also small, whence $R$ is also covered by some $E_v$. This yields a map \func{p}{E_v}{E_u \times E_u}, whose image $R$ is an equivalence relation, with quotient $X$.
\end{proof}

All in all, it seems more than just a good idea to restrict ones attention to categories with small maps that are exact, and, indeed, that is what we will do in our subsequent work. 

The problem that now arises is that exactness is not satisfied in our informal example, where \ct{E} is the category of classes in some set theory {\bf T} and the maps in \smallmap{S} are those maps whose fibres are sets in the sense of {\bf T}. For consider an equivalence relation
\[ R \subseteq X \times X \]
on the level of classes, so $R$ and $X$ are classes, and we need to see whether it has a quotient. The problem is that the standard construction does not work: the equivalence classes might indeed be genuine classes. Of course, we are only interested in the case where the mono $R \subseteq X \times X$ is small, but even then the equivalence classes might be large.

For some set theories ${\bf T}$ this problem can be overcome: for example, if ${\bf T}$ validates a global version of the axiom of choice, one could build a quotient by choosing representatives. Or if ${\bf T}$ is the classical set theory ${\bf ZF}$ (or some extension thereof) one could use an idea which is apparently due to Dana Scott: only take those elements from an equivalence class which have minimal rank. But in case $\bf T$ is some intuitionistic set theory, like ${\bf IZF}$ or ${\bf CZF}$, this will not work: in so far a constructive theory of ordinals can be developed at all, it will fail to make them linearly ordered. Indeed, we strongly suspect that for {\bf IZF} and {\bf CZF} the category of classes will not be exact.

We will solve this problem by showing that every category with small maps can ``conservatively'' be embedded in an exact category with small maps, and even in a universal way. We will call this its exact completion.

\section{Exact completion}

The notion of exact completion we will work with is the following:

\begin{defi}{Sexcompl}
The \emph{exact completion} of a category with small maps $(\ct{E}, \smallmap{S})$ is an exact category with small maps $(\overline{\ct{E}}, \overline{\smallmap{S}})$ together with a morphism 
\[ \func{\bf y}{(\ct{E}, \smallmap{S})}{(\overline{\ct{E}}, \overline{\smallmap{S}})}, \]
in such a way that precomposing with ${\bf y}$ induces for every exact category with small maps $(\ct{F}, \smallmap{T})$ an equivalence between morphisms from $(\overline{\ct{E}}, \overline{\smallmap{S}})$ to $(\ct{F}, \smallmap{T})$ and morphisms from $(\ct{E}, \smallmap{S})$ to $(\ct{F}, \smallmap{T})$.
\end{defi}

Clearly, exact completions (whenever they exist) are unique up to equivalence. The following is the main result of this section and we will devote the remainder of this section to its proof.

\begin{theo}{existexcompl}
The exact completion of a category with small maps $(\ct{E}, \smallmap{S})$ exists, and the functor \func{{\bf y}}{(\ct{E}, \smallmap{S})}{(\overline{\ct{E}}, \overline{\smallmap{S}})} has the following properties (besides being a morphism of categories of small maps):
\begin{enumerate}
\item it is full and faithful.
\item it is covering, i.e., for every $X \in \overline{\ct{E}}$ there is an object $Y \in \ct{E}$ together with a cover ${\bf y}Y \rTo X$.
\item it is bijective on subobjects.
\item $f \in \overline{\smallmap{S}}$ iff $f$ is covered by a map of the form ${\bf y}f'$ with $f' \in \smallmap{S}$.
\end{enumerate}
\end{theo}
Note that (1) and (4) imply that {\bf y} reflects small maps.

There is an extensive literature on exact completions of ordinary categories, which we will use to prove our result (\cite{menni00} is a useful source). The next theorem summarises what we need from this theory.

\begin{defi}{ordinexcompl}
Let \ct{C} be a positive regular category. By the \emph{exact completion} of \ct{C} (or the ex/reg-completion, or the exact completion of \ct{C} as a positive regular category) we mean a positive exact category (i.e., a pretopos) $\ct{E}_{ex/reg}$ together with a positive regular morphism \func{{\bf y}}{\ct{E}}{\ct{E}_{ex/reg}} such that precomposing with ${\bf y}$ induces for every pretopos \ct{F} an equivalence between pretopos morphisms from $\ct{E}_{ex/reg}$ to \ct{F} and positive regular morphisms from \ct{E} to \ct{F}.
\end{defi}

\begin{theo}{existexcomplofordcat}
The exact completion of a positive regular category \ct{C} exists, and the functor \func{{\bf y}}{\ct{C}}{\ct{C}_{\rm ex/reg}} has the following properties (besides being a morphism of positive regular categories):
\begin{enumerate}
\item it is full and faithful.
\item it is covering, i.e., for every $X \in \overline{\ct{E}}$ there is an object $Y \in \ct{E}$ together with a cover ${\bf y}Y \rTo X$.
\end{enumerate}
\end{theo}
\begin{proof}
See \cite{lackvitale01}.
\end{proof}

Note that because $\bf y$ is a full covering functor, every map $f$ in $\ct{C}_{\rm ex/reg}$ is covered by a map of the form ${\bf y}f'$ with $f' \in \ct{C}$. We will frequently exploit this fact.

As it happens, we can describe $\ct{C}_{ex/reg}$ explicitly. Objects of $\ct{C}_{ex/reg}$ are the equivalence relations in \ct{C}, which we will denote by $X/R$ when $R \subseteq X \times X$ is an equivalence relation. Morphisms from $X/R$ to $Y/S$ are \emph{functional relations}, i.e., subobjects $F \subseteq X \times Y$ satisfying the following statements in the internal logic of \ct{E}:
\begin{displaymath}
\begin{array}{l}
\exists y \, F(x, y), \\
xRx' \land ySy' \land F(x,y) \rightarrow F(x', y'), \\
F(x, y) \land F(x, y') \rightarrow ySy'.
\end{array}
\end{displaymath}
The functor \func{{\bf y}}{\ct{C}}{\ct{C}_{\rm ex/reg}} sends objects $X$ to their diagonals \func{\Delta_X}{X}{X \times X}.

One may then verify the following facts: when $R \subseteq X \times X$ is an equivalence relation in \ct{C}, its quotient in $\ct{C}_{ex/reg}$ is precisely $X/R$. When the equivalence relation already has a quotient $Q$ in \ct{C} this will be isomorphic to $X/R$ in $\ct{C}_{ex/reg}$. This means that an exact category is its own exact completion as a regular category, and the exact completion construction is idempotent.\footnote{This applies to the exact completion of a regular category \emph{as a regular category} only.}

\begin{lemm}{exreglemma}
Let $\ct{C}_{ex/reg}$ be the exact completion of a positive regular category \ct{C} and let {\bf y} be the standard embedding.
\begin{enumerate}
\item $\bf y$ induces an isomorphism between ${\rm Sub}(X)$ and ${\rm Sub}({\bf y}X)$ for every $X \in \ct{C}$.
\item When \ct{C} is Heyting, so is $\ct{C}_{ex/reg}$, and {\bf y} preserves this structure.
\end{enumerate}
\end{lemm} 
\begin{proof}
To prove 1, let \func{m}{D}{{\bf y}C'} be a mono in $\ct{C}_{ex/reg}$. Using that $\bf y$ is covering, we know that there is a cover \func{e}{{\bf y}C}{D}. Then, as $\bf y$ is full, there is a map $f \in \ct{C}$ such that ${\bf y}f = me$. Then we can factor $f = m'e'$ as a cover $e'$ followed by a mono $m'$. This factorisation is preserved by $\bf y$, so ${\bf y}f = {\bf y}m' {\bf y}e'$ factors ${\bf y}f$ as a cover followed by a mono. But as such factorisations are unique up to isomorphism, ${\bf y}m' = m$ as subobjects of ${\bf y}C'$.

When \ct{C} is Heyting, all pullback functors \func{({\bf y}f)^*}{{\rm Sub}({\bf y}X)}{{\rm Sub}({\bf y}Y)} for \func{f}{Y}{X} in \ct{C} have right adjoints by (1). As $\bf y$ is covering, every morphism $g$ in $\ct{C}_{ex/reg}$ is covered by an arrow ${\bf y}f$ with $f \in \ct{C}$:
\diag{ {\bf y}X \ar[d]_{{\bf y}f} \ar@{->>}[r]^q &  A \ar[d]^g \\
{\bf y}Y \ar@{->>}[r]_p & B.}
Now $\forall_g$ can be defined as $\exists_p \forall_{{\bf y}f} q^*$. To see this, let $K \subseteq A$ and $L \subseteq B$. That $g^*L \leq K$ implies $L \leq \exists_p \forall_{{\bf y}f} q^*K$, one shows directly using that $\exists_p p^* = 1$. The converse we show by using the internal logic. So let $a \in A$ be such that $g(a) \in L$. By assumption, there is an $y \in Y$ with $p(y) = g(a)$ such that for all $x \in ({\bf y}f)^{-1}(y)$, we have $q(x) \in K$. Because the square is a quasi-pullback, there is such an $x$ with $q(x) = a$. Therefore $a \in K$, and the proof is finished.
\end{proof}

From the description of the universal quantifiers in the proof of this lemma it follows that $\ct{E}_{ex/reg}$ is also the exact completion of \ct{E} as a positive Heyting category, when \ct{E} is a positive Heyting category. More precisely, when \ct{E} is a positive Heyting category and \ct{F} is a Heyting pretopos, precomposing with {\bf y} induces an equivalence between Heyting pretopos morphisms from $\ct{E}_{ex/reg}$ to \ct{F} and positive Heyting category morphisms from \ct{E} to \ct{F}.

We return to the original problem of constructing the exact completion of a category with small maps $(\ct{E}, \smallmap{S})$. As suggested by the statement of \reftheo{existexcompl}, we single out the following class of maps $\overline{\smallmap{S}}$ in $\ct{E}_{ex/reg}$:
\begin{eqnarray*}
g \in \overline{\smallmap{S}} & \Leftrightarrow & g \mbox{ is covered by a morphism of the form } {\bf y}f \mbox{ with } f \in \smallmap{S}.
\end{eqnarray*}
In the next two lemmas, we show that this class of maps satisfies the axioms {\bf (A1-8)} for a class of small maps in $\ct{E}_{ex/reg}$. The proof is very similar to the argument we gave to show that $\smallmap{S}^{\rm cov}$ defines a class of small maps for a class of display maps $\smallmap{S}$ in Section 2.3.

\begin{lemm}{2nduseofcoll}
Any two maps \func{f}{Y}{X} and \func{g}{Z}{Y} belonging to $\overline{\smallmap{S}}$ fit into a diagram of the form
\diag{ {\bf y} Z' \ar@{->>}[r] \ar[d]_{{\bf y}g'} & Z \ar[d]^g \\
{\bf y} Y' \ar@{->>}[r] \ar[d]_{{\bf y} f'} & Y \ar[d]^{f} \\
{\bf y} X' \ar@{->>}[r] & X, }
where both squares are covering squares and $f'$ and $g'$ belong to \smallmap{S}.
\end{lemm}
\begin{proof}
By definition of $\overline{\smallmap{S}}$, $g$ and $f$ fit a diagram of the form
\diag{ & {\bf y}D \ar[dl]_{{\bf y}g_0} \ar@{->>}[dr] & \\
{\bf y}C \ar@{->>}[dr] & & Z \ar[dl]^g \\
& Y \ar[dr]^f & \\
{\bf y}B \ar@{->>}[ur] \ar[dr]_{{\bf y}f_0} & & X \\
& {\bf y}A, \ar@{->>}[ur] & }
with $f_0, g_0 \in \smallmap{S}$ and the squares covering. By computing the pullback ${\bf y}B \times_Y {\bf y}C$ and covering this with ${\bf y}E \rTo {\bf y}B \times_Y {\bf y}C$, we obtain a diagram to which we can apply collection (in \ct{E}), resulting in:
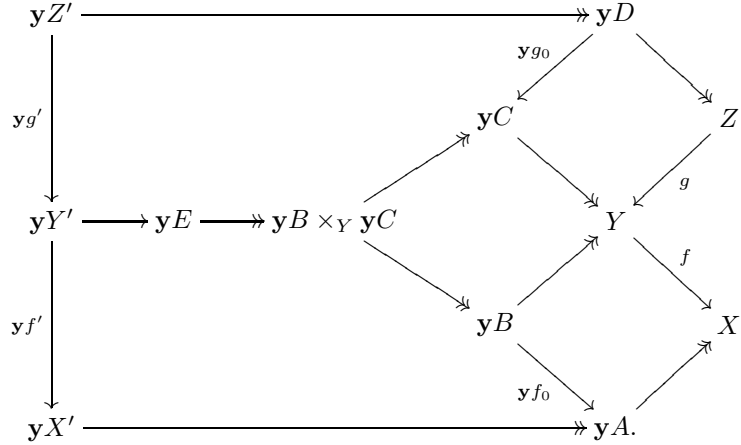
\diag{ {\bf y}Z' \ar[dd]_{{\bf y}g'} \ar@{->>}[rrrr] & & & & {\bf y}D \ar[dl]_{{\bf y}g_0} \ar@{->>}[dr] & \\
& & & {\bf y}C \ar@{->>}[dr] & & Z \ar[dl]^g \\
{\bf y}Y' \ar[r] \ar[dd]_{{\bf y}f'} & {\bf y}E \ar@{->>}[r] & {\bf y}B \times_Y {\bf y}C \ar@{->>}[dr] \ar@{->>}[ur] & & Y \ar[dr]^f & \\
& & & {\bf y}B \ar@{->>}[ur] \ar[dr]_{{\bf y}f_0} & & X \\
{\bf y}X' \ar@{->>}[rrrr] & & & & {\bf y}A. \ar@{->>}[ur] & }
Finally, the map ${\bf y} g'$ is obtained by pulling back $g_0$, so also this map belongs to \smallmap{S}. This finishes the proof.
\end{proof}

\begin{lemm}{propsbar}
The class of maps $\overline{\smallmap{S}}$ defined above satisfies axioms {\bf (A1-8)}.
\end{lemm}
\begin{proof}
The class of maps $\overline{\smallmap{S}}$ is closed under covered maps by \reflemm{covsqpasting}, so {\bf (A2)} and {\bf (A6)} follow immediately. The axiom {\bf (A3)} follows from \reflemm{covsqpasting} as well, combined with the fact that {\bf y} preserves the positive structure. {\bf (A4)} follows because {\bf y} preserves the lextensive structure, and {\bf (A5)} follows from the previous lemma. Verifying the other axioms is more involved.

{\bf (A1)}: Assume $f \in \ct{E}_{ex/reg}$ can be obtained by pullback from a map $g \in \overline{\smallmap {S}}$. Then $f$ and $g$ fit into a diagram as follows:
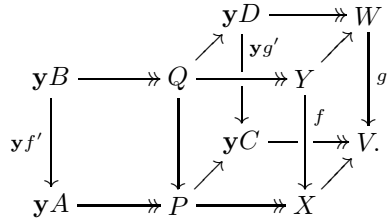
\begin{displaymath}
\xymatrix@!0{
& & & {\bf y}D  \ar@{->>}[rr] \ar'[d]^{{\bf y}g'}[dd] & & W \ar[dd]^g \\
{\bf y}B \ar@{->>}[rr] \ar[dd]_{{\bf y}f'} & & Q \ar[ur]\ar@{->>}[rr]\ar[dd] & & Y \ar[dd]^(.3)f \ar[ur]  \\
& & & {\bf y}C \ar@{->>}'[r][rr] & & V. \\
{\bf y}A \ar@{->>}[rr] & & P \ar@{->>}[rr] \ar[ur] & & X \ar[ur] }
\end{displaymath}
The picture has been constructed in several steps. First, we obtain at the back of the cube a covering square involving a map ${\bf y}g'$ with $g' \in \smallmap{S}$ by definition of $\overline{\smallmap{S}}$. Next, this square is pulled back along the map $X \rTo V$, making the front covering as well (by \reflemm{covsqpasting}), and the other faces pullbacks. Finally, we obtain a cover ${\bf y}A \rTo P$, using that {\bf y} is covering, and ${\bf y}B$ by pullback, using that {\bf y} preserves pullbacks. By pullback stability of \smallmap{S}, $f' \in \smallmap{S}$, so that $f \in \overline{\smallmap{S}}$.

{\bf (A7)}: Let $\func{f}{Y}{X} \in \overline{\smallmap{S}}$ and a cover $Z \rTo Y$ be given. We obtain a diagram as follows, again constructed in several steps.
\begin{displaymath}
\xymatrix@!0{
& & & & & Z  \ar@{->>}[rr]  & & Y \ar[dd]^f \\
{\bf y}D \ar[dd]_{{\bf y}g'} \ar[rr] & & {\bf y}E \ar@{->>}[rr]  & & P \ar[ur]\ar@{->>}[rr] & & {\bf y}B \ar[dd]^(.4){{\bf y}f'} \ar[ur]  \\
& & & & & & & X. \\
{\bf y}C \ar@{->>}[rrrrrr] & &  & &  & & {\bf y}A \ar[ur] }
\end{displaymath}
First, we find a map ${\bf y}f'$ with $f' \in \smallmap{S}$ covering $f$. Next, we obtain the object $P$ by pullback, and we let ${\bf y}E$ be an object covering $P$. Finally, we apply Collection in \ct{E} to $f'$ and the cover $E \rTo B$ to get a map $g' \in \smallmap{S}$ covering $f'$ in \ct{E}. As covering squares are preserved by {\bf y}, it follows that ${\bf y}g'$ covers ${\bf y}f'$, and hence also $f$.

{\bf (A8)}: Let $\func{f}{Y}{X}$ be a map belonging to $\overline{\smallmap{S}}$, and let $A$ be an $\overline{\smallmap{S}}$-bounded subobject of $Y$. Using the previous lemma, we obtain a diagram 
\diag{ {\bf y} A' \ar@{->>}[r] \ar[d]_{{\bf y}i'} & A \ar@{ >->}[d]^i \\
{\bf y} Y' \ar@{->>}[r]^q \ar[d]_{{\bf y} f'} & Y \ar[d]^{f} \\
{\bf y} X' \ar@{->>}[r]_p & X, }
with $i', f' \in \smallmap{S}$ and both squares covering. As \smallmap{S} satisfies the quotient axiom {\bf (A6)}, we may actually assume that the top square is a pullback and $i'$ is monic. Observe that the proof of \reflemm{exreglemma} yields the formula $\forall_f (i) = \exists_p \forall_{{\bf y}f'} ({\bf y}i')$. But $\forall_{{\bf y}f'} ({\bf y}i')$ is an \smallmap{S}-bounded subobject of ${\bf y}X'$ as {\bf (A8)} holds for \smallmap{S}, and then $\exists_p \forall_{{\bf y}f'} ({\bf y}i')$ is an $\overline{\smallmap{S}}$-bounded subobject of $X$ by Descent for $\overline{\smallmap{S}}$.
\end{proof}

The problem with the pair $(\ct{E}_{ex/reg}, \overline{\smallmap{S}})$ is that it does not satisfy axiom {\bf (A9)} (in general). Therefore, call an object $X$ \emph{separated} relative to a class of maps \smallmap{T}, when the diagonal $X \rTo X \times X$ belongs to \smallmap{T}. We will write ${\rm Sep}_{\smallmap{T}}(\ct{E})$ for the full subcategory of \ct{E} consisting of the separated objects. Using this notation we define
\[ \overline{\ct{E}} = {\rm Sep}_{\overline{\smallmap{S}}}(\ct{E}_{ex/reg}). \]

\begin{lemm}{propofebar}
$(\overline{\ct{E}}, \overline{\smallmap{S}})$ is an exact category with small maps.
\end{lemm}
\begin{proof}
Essentially a routine exercise. $\overline{\ct{E}}$ is a Heyting category, because the terminal object is separated, and separated objects are closed under products and subobjects. Separated objects are also closed under sums, so that $\overline{\ct{E}}$ is a positive Heyting category.

In showing that $\overline{\smallmap{S}}$ is a class of small maps, the only difficulty is proving that it satisfies the Collection axiom {\bf (A7)}. But note that in the proof of the previous lemma, while showing that $\overline{\smallmap{S}}$ satisfies the axiom {\bf (A7)} in $\ct{E}_{ex/reg}$, we showed a bit more: we actually proved that, in the notation we used there, the map covering $f$ could be chosen to be of the form ${\bf y}g'$. But this is a map between separated objects, since all objects of the form ${\bf y}X$ are separated.

To prove that $(\overline{\ct{E}}, \overline{\smallmap{S}})$ is exact, it suffices to show that the quotient \func{q}{X}{Q} in $\ct{E}_{ex/reg}$ of an $\overline{\smallmap{S}}$-bounded equivalence relation $R \subseteq X \times X$ is separated. That follows from Descent for $\overline{\smallmap{S}}$ in $\ct{E}_{ex/reg}$, as the following square is a pullback:
\diag{ R \ar[rr] \ar@{ >->}[d] & & Q \ar@{ >->}[d] \\
X \times X \ar@{->>}[rr]_{q \times q} & & Q \times Q. }
\end{proof}

Let's see to what extent we have established \reftheo{existexcompl}. Since objects of the form ${\bf y}X$ are separated, the morphism \func{{\bf y}}{\ct{E}}{\ct{E}_{ex/reg}} factors through $\overline{\ct{E}}$. It is clear that {\bf y} considered as functor $\ct{E} \rTo \overline{\ct{E}}$ is still a morphism of positive Heyting categories satisfying items (1) and (2) from \reftheo{existexcompl}. It is immediate from the definition of $\overline{\smallmap{S}}$ that it preserves small maps, so that
\[ \func{\bf y}{(\ct{E}, \smallmap{S})}{(\overline{\ct{E}}, \overline{\smallmap{S}})} \]
is indeed a morphism of categories with small maps. Furthermore, it also satisfies item (3), because {\bf y} it is bijective on subobjects by \reflemm{exreglemma}, and the definition of $\overline{\smallmap{S}}$ was made so as to make it satisfy item (4) as well.

Therefore, to complete the proof of \reftheo{existexcompl}, it remains to show the universal property of $(\overline{\ct{E}}, \overline{\smallmap{S}})$. For this we use:
\begin{lemm}{excidempot}
For an exact category with small maps $(\ct{F}, \smallmap{T})$, we have that
\[ (\ct{F}, \smallmap{T}) \cong (\overline{\ct{F}}, \overline{\smallmap{T}}). \]
\end{lemm}
\begin{proof}
It suffices to point out that \func{{\bf y}}{\ct{F}}{\overline{\ct{F}}} is essentially surjective on objects. We know that every object in $\overline{\ct{F}}$ arises as a quotient $X/R$ of an equivalence relation $R \subseteq X \times X$ in \ct{E}. But we can say more: $X/R$ is $\overline{\smallmap{T}}$-separated, so the equivalence relation $R \subseteq X \times X$ is $\overline{\smallmap{T}}$-bounded, and therefore also \smallmap{T}-bounded, because {\bf y} reflects small maps. So a quotient $Q$ of this equivalence relation already exists in \ct{F}, and as this is preserved by {\bf y}, we get that ${\bf y}Q \cong R/X$.
\end{proof}

So let $(\ct{F}, \smallmap{T})$ be an exact category with small maps, and \func{F}{(\ct{E}, \smallmap{S})}{(\ct{F}, \smallmap{T})} be a morphism of categories with small maps. Consider the exact completion $\ct{F}_{ex/reg}$ of \ct{F}, together with \func{{\bf y}}{\ct{F}}{\ct{F}_{ex/reg}}. Then there is an exact morphism \func{\overline{F}}{\ct{E}_{ex/reg}}{\ct{F}_{ex/reg}} such that ${\bf y}F \cong \overline{F} {\bf y}$, by the universal property of $\ct{E}_{ex/reg}$. This morphism $\overline{F}$ also preserves the positive and Heyting structure of $\ct{E}_{ex/reg}$, and, moreover, sends morphisms in $\overline{\smallmap{S}}$ to those in $\overline{\smallmap{T}}$. Therefore $\overline{F}$ restricts to a functor between the separated objects in $\ct{E}_{ex/reg}$ and those in $\ct{F}_{ex/reg}$, that is, a functor between categories with small maps from $(\overline{\ct{E}}, \overline{\smallmap{S}})$ to $(\ct{F}, \smallmap{T})$. This completes the proof of \reftheo{existexcompl}.

\begin{rema}{expldescrexc}
The question arises as to whether we can describe the category $\overline{\ct{E}}$ more concretely, i.e., if we can identify those objects in $\ct{E}_{ex/reg}$ that belong to $\overline{\ct{E}}$. As was implicitly shown in the proof of \reflemm{excidempot}, these are precisely the bounded equivalence relations.
\end{rema}

\begin{rema}{excomplslicing}
An important property of exact completions is their stability under slicing. By this we mean that for any category with small (or display) maps $(\ct{E}, \smallmap{S})$ and object $X$ in \ct{E},
\[ (\overline{\ct{E}/X}, \overline{\smallmap{S}/X}) \cong (\overline{\ct{E}}/{\bf y}X, \overline{\smallmap{S}}/{\bf y}X). \]
A formal proof is left to the reader.
\end{rema}

\begin{rema}{excfordisplaymaps}
When we combine \reftheo{existexcompl} with our earlier work on display maps, we obtain the following result:
\begin{coro}{existexcompldispl}
For every category with display maps $(\ct{E}, \smallmap{S})$ there exists an exact category with small maps $(\ct{F}, \smallmap{T})$ together with a functor \func{\bf y}{\ct{E}}{\ct{F}} of positive Heyting categories with the following properties:
\begin{enumerate}
\item it is full and faithful.
\item it is covering.
\item it is a bijection on subobjects.
\item $f \in \overline{\smallmap{S}}$ iff $f$ is covered by a map of the form ${\bf y}f'$ with $f' \in \smallmap{S}$.
\end{enumerate}
\end{coro}
For $(\ct{F}, \smallmap{T})$ we can simply take the exact completion $(\overline{\ct{E}}, \overline{\smallmap{S}^{\rm cov}})$ of $(\ct{E}, \smallmap{S}^{\rm cov})$. By abuse of terminology and notation, we will refer to this category as the \emph{exact completion of the category with display maps} $(\ct{E}, \smallmap{S})$, and denote it by $(\overline{\ct{E}}, \overline{\smallmap{S}})$ as well. 

To abuse terminology even further, we will call a category with display maps $(\ct{E}, \smallmap{S})$ \emph{(bounded) exact}, when $(\ct{E}, \smallmap{S}^{\rm cov})$ is a (bounded) exact category with small maps. Note that for an exact category with display maps, $(\overline{\ct{E}}, \overline{\smallmap{S}}) = (\ct{E}, \smallmap{S}^{\rm cov})$.

Actually, as is not too hard to see using the results obtained in this section, the properties of $(\ct{F}, \smallmap{T})$ and {\bf y} formulated in the Corollary determine these uniquely up to equivalence. \emph{A fortiori}, the same remark applies to \reftheo{existexcompl}.
\end{rema}

\section{Stability properties of axioms for classes of small maps}

In this -- rather technical -- section of the paper we want to show, among other things, the stability under exact completion of additional axioms for a class of small maps. The importance of this resides in the fact that many of these axioms are needed to model the axioms of {\bf IZF} and {\bf CZF}. So this section makes sure that in studying these set theories we can safely restrict our attention to exact categories with small maps.

We should point out that we are not able to show the stability of all the axioms we mentioned in Section 3 under exact completion. In fact, we conjecture that {\bf ($\Pi$S)} and {\bf (WS)} are not. But, fortunately, these axioms are not necessary for modelling either {\bf IZF} or {\bf CZF}.

But for those axioms for which we can show stability, we will actually be able to show something slightly stronger: we will show that their validity is preserved by the exact completion $(\overline{\ct{E}}, \overline{\smallmap{S}})$, assuming only that $(\ct{E}, \smallmap{S})$ is a category with \emph{display maps} (see \refrema{excfordisplaymaps}). It is in this form we will need the results from this section in our subsequent work on realisability, for in that case the appropriate category with small maps is constructed using display maps (our paper \cite{bergmoerdijk07a} gives the idea). 

So in this section, $(\ct{E}, \smallmap{S})$ will be a category with display maps, unless explicitly stated otherwise.

Simultaneously, we will discuss which of the axioms are inherited by covered maps (i.e., by $\smallmap{S}^{\rm cov}$ from $\smallmap{S}$). The reason why we discuss this question in parallel with the other one is that the proofs of their stability (in case they are stable) are almost identical. So what we will typically do is show stability under exact completion and then point out that an almost identical proof shows stability under covered maps. In some cases the argument only works for exact categories with display maps. When this is the case, we will point this out as well.

\subsection{Representability}

\begin{prop}{presreprcov}
Let $(\ct{E}, \smallmap{S})$ a category with display maps. Then a representation for $\smallmap{S}$ is also a representation for $\smallmap{S}^{\rm cov}$. Indeed, \smallmap{S} is representable iff $\smallmap{S}^{\rm cov}$ is. 
\end{prop}
\begin{prop}{presrepr}
Let \func{{\bf y}}{(\ct{E}, \smallmap{S})}{(\overline{\ct{E}}, \overline{\smallmap{S}})} be the exact completion of a category with display maps. Then \smallmap{S} is representable iff $\overline{\smallmap{S}}$ is. Moreover, {\bf y} preserves and reflects representations.
\end{prop}
We omit the proofs, as by now these should be routine. The only insight they require is that a small map covering a representation is again a representation.

\subsection{Separation}

The following two propositions are even easier to prove:
\begin{prop}{presMcov}
Let $(\ct{E}, \smallmap{S})$ a category with display maps. When $\smallmap{S}$ satisfies {\bf (M)}, then so does $\smallmap{S}^{\rm cov}$.
\end{prop}

\begin{prop}{presM}
Let \func{{\bf y}}{(\ct{E}, \smallmap{S})}{(\overline{\ct{E}}, \overline{\smallmap{S}})} be the exact completion of a category with display maps. When \smallmap{S} satisfies {\bf (M)}, then so does $\overline{\smallmap{S}}$.
\end{prop}

\subsection{Power types}

In this subsection, we give proofs for the stability of {\bf (PE)} and {\bf (PS)} under exact completion and covered maps. They all rely on the following lemma:

\begin{lemm}{redlemmforpowobj}
Let \func{{\bf y}}{(\ct{E}, \smallmap{S})}{(\overline{\ct{E}}, \overline{\smallmap{S}})} be the exact completion of a category with display maps. When $\spower X$ is the power object for $X$ in \ct{E}, then ${\bf y}\spower X$ is the power object for ${\bf y}X$ in $\overline{\ct{E}}$.
\end{lemm}
\begin{proof}
From now on, we will drop occurences of {\bf y} in the proofs. 

For the purpose of showing that $\spower X$ in \ct{E} has the universal property of the power class object of $X$ in $\overline{\ct{E}}$, let $U \subseteq X \times I \rTo I$ be an $\overline{\smallmap{S}}$-displayed $I$-indexed family of subobjects of $X$. We need to show that there is a unique map \func{\rho}{I}{\spower X} such that $(\id \times \rho)^* \in_X = U$.

Since $U \rTo I \in \overline{\smallmap{S}}$, there is a map $V \rTo J \in \smallmap{S}$ such that the outer rectangle in
\diag{V \ar[d]_f \ar@{->>}[r] & U \ar@{ >->}[d] \\
X \times J \ar[r] \ar[d] & X \times I \ar[d] \\
J \ar@{->>}[r]_p & I,} 
is a covering square. Now also $\func{f}{V}{X \times J} \in \smallmap{S}$ by \reflemm{composlemma}. By replacing $f$ by its image if necessary and using the axiom {\bf (A10)}, we may assume that the top square (and hence the entire diagram) is a pullback and $f$ is monic.

So there is a classifying map \func{\sigma}{J}{\spower X} in \ct{E}, by the universal property of $\spower X$ in \ct{E}. This map $\sigma$ coequalises the kernel pair of $p$, again by the universal property of $\spower X$ and $\in_X$.  Therefore there is a unique map \func{\rho}{I}{\spower X} such that $\rho p = \sigma$:
\diag{V \ar@{ >->}[d]_f \ar@{->>}[r] & U \ar@{ >->}[d] \ar[r] & \in_X \ar@{ >->}[d] \\
X \times J \ar@{->>}[r] \ar[d] & X \times I \ar[d] \ar[r] & X \times \spower X \ar[d] \\
J \ar@{->>}[r]^p \ar@/_/@<-1ex>[rr]_{\sigma} & I \ar[r]^{\rho} & \spower X. } 
The desired equality $(\id \times \rho)^* \in_X = U$ follows from \reflemm{pspastinregcat}, and the uniqueness of $\rho$ follows from the fact that $p$ is epic.
\end{proof}
\begin{prop}{presPEcov}
Let $(\ct{E}, \smallmap{S})$ be a category with display maps \smallmap{S}. When \smallmap{S} satisfies {\bf (PE)}, then so does $\smallmap{S}^{\rm cov}$. Indeed, the power class objects for both classes of maps coincide.
\end{prop}
\begin{proof}
The proof of the lemma above can be copied verbatim, making the obvious minor changes: in particular, replacing $\overline{\ct{E}}$ by $\ct{E}$ and $\overline{\smallmap{S}}$ by $\smallmap{S}^{\rm cov}$.
\end{proof}
\begin{prop}{presPE}
Let \func{{\bf y}}{(\ct{E}, \smallmap{S})}{(\overline{\ct{E}}, \overline{\smallmap{S}})} be the exact completion of a category with display maps \smallmap{S}. When \smallmap{S} satisfies {\bf (PE)}, then so does $\overline{\smallmap{S}}$. Moreover, {\bf y} preserves power class objects.
\end{prop}
\begin{proof}
Let $Y$ be an arbitrary object in $\overline{\ct{E}}$. Since {\bf y} is covering, there is an $X \in \ct{E}$ together with a cover
\[ \func{q}{X}{Y}. \]
From \reflemm{redlemmforpowobj} we learn that $X$ has a powerobject $\spower X$ in $\overline{\ct{E}}$. On this object we can define the following equivalence relation:
\begin{eqnarray*}
\alpha \sim \beta & \Leftrightarrow & q\alpha = q\beta \\
& \Leftrightarrow & \forall a \in \alpha \, \exists b \in \beta : qa = qb \land \forall b \in \beta \, \exists a \in \alpha : qa = qb.
\end{eqnarray*}
We claim that the quotient of $\spower X$ with respect to this $\overline{\smallmap{S}}$-bounded equivalence relation, which we will write $\spower Y$, is indeed the power object of $Y$. 

We first need to construct an $\overline{\smallmap{S}}$-displayed $\spower Y$-indexed family of subobjects of $Y$: it is defined as the image of $\in_X$ along $X \times \spower X \rTo Y \times \spower Y$. Then, since the entire diagram in
\diag{ \in_X \ar@{ >->}[d] \ar@{->>}[r] & \in_Y \ar@{ >->}[d] \\
X \times \spower X \ar[d] \ar[r] & Y \times \spower Y \ar[d] \\
\spower X \ar@{->>}[r] & \spower Y}
is a covering square, $\in_Y \rTo \spower Y \in \overline{\smallmap{S}}$.

It remains to verify the universal property of $\in_Y$. So let $U \subseteq Y \times I$ be an $\overline{\smallmap{S}}$-displayed $I$-indexed family of subobjects of $Y$. We need to find a map \func{\rho}{I}{\spower Y} such that $(\id \times \rho)^* \in_Y = U$. Pulling back $U \subseteq Y \times I$ along $X \times I \rTo Y \times I$, we obtain a subobject $q^* U \subseteq X \times I$. Then we use the Collection axiom for $\overline{\smallmap{S}}$ to obtain a covering square of the form
\diag{V \ar[r] \ar[dd] & q^*U \ar@{ >->}[d] \ar@{->>}[r] & U \ar@{ >->}[d] \\
& X \times I \ar@{->>}[r] & Y \times I \ar[d] \\
J \ar@{->>}[rr]_p & & I,}
with $V \rTo J \in \overline{\smallmap{S}}$. By considering the diagram
\diag{V \ar@/_/[ddr] \ar@/^/[drr] \ar[dr] \\
& X \times J \ar[r] \ar[d] & X \times I \ar[d] \\
& J \ar[r] & I,}
we see that the image $V'$ of $V$ in $X \times J$ defines an $\overline{\smallmap{S}}$-displayed $J$-indexed family of subobjects of $X$, and therefore a morphism \func{\sigma}{J}{\spower X}. We now claim that the composite
\diag{ m: J \ar[r]^{\sigma} & \spower X \ar[r] & \spower Y }
coequalises the kernel pair of the cover \func{p}{J}{I}. This follows from the fact that $m(j)$ equals $m(j')$ in case the images of $V_j \rTo X \rTo Y$ and $V_{j'} \rTo X \rTo Y$ are the same. But as these images are precisely $U_{p(j)}$ and $U_{p(j')}$, this happens in particular whenever $p(j) = p(j')$.  Therefore we obtain a morphism \func{\rho}{I}{\spower Y} such that $\rho p = \sigma$. The proof that it has the desired property, and is the unique such, is left to the reader.
\end{proof}

\begin{prop}{PSstable}
Let \func{{\bf y}}{(\ct{E}, \smallmap{S})}{(\overline{\ct{E}}, \overline{\smallmap{S}})} be the exact completion of a category with display maps. When \smallmap{S} satisfies {\bf (PS)}, then so does $\overline{\smallmap{S}}$.
\end{prop}
\begin{proof}
Consider a map $\func{f}{B}{A}$ in $\overline{\smallmap{S}}$. There is a $\func{g}{Y}{X}$ in $\smallmap{S}$ such that
\diag{ Y \ar[d]_g \ar@{->>}[r] & B \ar[d]^f \\
X \ar@{->>}[r]_p & A }
is a covering square. This we can decompose as follows:
\diag{ Y \ar[dr]_g \ar@{->>}[r]^e & p^* B \ar@{->>}[r] \ar[d]^{p^*f} & B \ar[d]^f \\
& X \ar@{->>}[r]_p & A. }
From the validity of {\bf (PS)} for \smallmap{S} it follows that $\slspower{X}(g) \rTo X$ is $\smallmap{S}$-small in \ct{E}, and also $\overline{\smallmap{S}}$-small in $\overline{\ct{E}}$ by \reflemm{redlemmforpowobj} (and \refrema{excomplslicing}). As \refprop{collandPE} implies that \func{\slspower{X}(e)}{\slspower{X} (g)}{\slspower{X}(p^*f)} is a cover, we see that $\slspower{X}(p^*f) \rTo X$ belongs to $\overline{\smallmap{S}}$. Hence the same holds for $\slspower{A}(f) \rTo A$ by Descent.
\end{proof}

The same argument shows:
\begin{prop}{presPScov}
Let $(\ct{E}, \smallmap{S})$ be a category with display maps \smallmap{S}. When \smallmap{S} satisfies {\bf (PS)}, then so does $\smallmap{S}^{\rm cov}$. 
\end{prop}

\subsection{Function types}

As we already announced in the introduction to this section, we will not be able to show stability of the axiom {\bf ($\Pi$S)}. The difficulty is that the morphism in $\overline{\ct{E}}$ are functional relations in \ct{E}. In fact, for this reason we actually conjecture that the axiom {\bf ($\Pi$S)} is not stable.

On the other hand, by generalising Theorem I.3.1 in \cite{joyalmoerdijk95}, we can show the stability of the axiom {\bf ($\Pi$E)} for representable classes of display maps.
\begin{prop}{presPiE}
Let \func{{\bf y}}{(\ct{E}, \smallmap{S})}{(\overline{\ct{E}}, \overline{\smallmap{S}})} be the exact completion of a category with display maps \smallmap{S}. If \smallmap{S} is representable and satisfies {\bf ($\Pi$E)}, then $\overline{\smallmap{S}}$ satisfies {\bf (PE)} as well as {\bf ($\Pi$E)}.
\end{prop}
\begin{proof}
As the validity of {\bf (PE)} implies that of {\bf ($\Pi$E)} by \reflemm{PEimpliesPiE}, we only need to construct power class objects in $\overline{\ct{E}}$. And it suffices to do this for the objects $X \in \ct{E}$, for the general case will then follow as in the proof of \refprop{presPE}.

In \ct{E}, the class \smallmap{S} has a representation \func{\pi}{E}{U}, which is exponentiable in \ct{E}, since we are assuming {\bf ($\Pi$E)} for \smallmap{S}. Therefore we can build in \ct{E} the object
\[ P_{\pi}(X) = \{ u \in U, \func{t}{E_u}{X} \}, \]
together with the equivalence relation
\begin{eqnarray*}
(u, t) \sim (u, t') & \Leftrightarrow & {\rm Im}(t) = {\rm Im}(t') \\
& \Leftrightarrow & \forall e \in E_u \, \exists e' \in E_{u'} \, te = t'e' \land \forall e' \in E_{u'} \, \exists e \in E_{u} \, te = t'e'.
\end{eqnarray*}
This equivalence relation is \smallmap{S}-bounded, so also $\overline{\smallmap{S}}$-bounded in $\overline{\ct{E}}$. Therefore we can take its quotient in $\overline{\ct{E}}$, which we will write as $\spower X$. We claim it is the power class object of $X$ in $\overline{\ct{E}}$.

To show this, we first have to define an $\overline{\smallmap{S}}$-displayed family of subobjects of $X$ in $\overline{\ct{E}}$. Let $L \subseteq X \times P_{\pi}(X)$ be defined by
\[ (x, u, t) \in L \Leftrightarrow \exists e \in E_u \, te = x. \]
Then define $\in_X$ as the image of $L$ along $X \times P_{\pi}(X) \rTo X \times \spower(X)$:
\diag{ L \ar@{ >->}[d] \ar@{->>}[r] & \in_X \ar@{ >->}[d] \\
X \times P_{\pi}(X) \ar[d] \ar[r] & X \times \spower X \ar[d] \\
P_{\pi}(X) \ar[r] & \spower X.}
Since 
\[ (x, u, t) \in L \land (u, t) \sim (u', t') \Rightarrow (x, u', t') \in L, \]
the top square in the above diagram is a pullback, and therefore the entire diagram is a pullback. So the fact that $\in_X \rTo \spower X$ belongs to $\overline{\smallmap{S}}$ follows from the fact that $L \rTo P_{\pi} X$ belongs to $\smallmap{S}$.

To check the universal property of $\spower X$ with $\in_X$, let $U \subseteq X \times I$ be an $\overline{\smallmap{S}}$-displayed family of subobjects of $X$ in $\overline{\ct{E}}$. We need to find a map \func{\rho}{I}{\spower X} such that $(\id \times \rho)^* \in_X = U$. 

As $U \rTo I \in \overline{\smallmap{S}}$, it fits into a covering square with $V \rTo J \in \smallmap{S}$ as follows:
\diag{ V \ar[d] \ar@{->}[r] & U \ar@{ >->}[d] \\
X \times J \ar[d] \ar[r] & X \times I \ar[d] \\
J \ar@{->>}[r]_q &  I.}
The fact that $V \rTo J$ belongs to \smallmap{S} means that for every $j \in J$ there is a morphism \func{\phi_j}{V_j}{X}, where $V_j$ is \smallmap{S}-small. Then, since $\pi$ is a representation, the following statement holds in \ct{E}:
\[ \forall j \in J \, \exists u \in U, \func{p}{E_u}{V_j} \, ( p \mbox{ is a cover}), \]
and hence the following as well:
\[ \forall j \in J \, \exists u \in U, \func{t}{E_u}{X} \,  ({\rm Im}(t) = {\rm Im}(\phi_j)) \]
(for $t$ take the composite of $p$ and $\phi_j$). Defining $G \subseteq J \times P_{\pi}(X)$ by
\[ (j, u, t) \in G \Leftrightarrow {\rm Im}(t) = {\rm Im}(\phi_j), \]
we can write this as 
\[ \forall j \in J \, \exists (u, t) \in P_{\pi}(X) \, ((j, u, t) \in G). \]
Since clearly 
\[ (j, u, t), (j, u', t') \in G \Rightarrow (u, t) \sim (u', t'), \]
$G$ defines the graph of a function \func{\sigma}{J}{\spower X}. This $\sigma$ coequalises the kernel pair of the cover \func{q}{J}{I}, for the following reason. The righthand arrow in the above diagram defines for every $i \in I$ a morphism \func{\psi_i}{U_i}{X}, and the fact that the entire diagram is a quasi-pullback means that
\[ {\rm Im}(\phi_j) = {\rm Im}(\psi_{qj}). \]
Therefore ${\rm Im}(\phi_j) = {\rm Im}(\phi_k)$, whenever $qj = qk$, or:
\[ (j, u, t) \in G, qj = qk \Rightarrow (k, u, t) \in G. \]
So $\sigma$ coequalises the kernel pair of $q$, and we find a morphism \func{\rho}{I}{\spower X} such that $\rho q = \sigma$. We leave the proof that it has the required property, and is the unique such, to the reader.
\end{proof}

An immediate corollary of this proposition is the following result, which is essentially Theorem I.3.1 on page 16 of \cite{joyalmoerdijk95}, but derived here using bounded exactness only.
\begin{coro}{PiEimpliesPE}
Let $(\ct{E}, \smallmap{S})$ be an exact category with a representable class of display maps \smallmap{S} satisfying {\bf($\Pi$E)}. Then \smallmap{S} also satisfies {\bf (PE)}. Moreover, there exists a natural transformation
\[ \func{\tau_X}{P_{\pi} X = \sum_{u \in U} X^{E_u}}{\spower X} \]
which is componentwise a cover.
\end{coro}

We will now briefly discuss the stability of {\bf ($\Pi$E)} and {\bf ($\Pi$S)} under covered map. Again, stability of {\bf ($\Pi$S)} seems problematic, while for {\bf ($\Pi$E)} we have the following result:
\begin{prop}{PiEstableundercov}
Let $(\ct{E}, \smallmap{S})$ be an exact category with a class of display maps \smallmap{S}. When \smallmap{S} satisfies {\bf($\Pi$E)}, then so does $\smallmap{S}^{\rm cov}$.
\end{prop}
\begin{proof} We omit a proof, but it could go along the lines of Lemma I.1.2 on page 9 of \cite{joyalmoerdijk95}, all the time making sure we use bounded exactness only.
\end{proof}

\subsection{Inductive types}

The situation for the axioms for W-types is the same as that for the $\Pi$-types. We conjecture that {\bf (WS)} is not a stable under exact completion, like {\bf ($\Pi$S)}, while the axiom {\bf (WE)} is stable under exact completion for representable classes of display maps. It is by no means easy to establish this, and the remainder of this subsection will be devoted to a proof. (The results that will now follow are, in fact, variations on results of the first author, published in \cite{berg05}.)

We first prove the following characterisation theorem:
\begin{theo}{charthmWtypes}
Let \ct{E} be a category with a class of small maps \smallmap{S} satisfying {\bf (PE)}. Assume that \func{f}{B}{A} is a small map. The following are equivalent for a $P_{f}$-algebra $(W, \func{{\rm sup}}{P_{f}(W)}{W})$:
\begin{enumerate}
\item $(W, {\rm sup})$ is a W-type for $f$.
\item The structure map ${\rm sup}$ is an isomorphism and $W$ has no proper $P_f$-subalgebras in $\ct{E}$.
\item The structure map ${\rm sup}$ is an isomorphism and $X^{*}W$ has no proper $P_{X^{*}f}$-subalgebras in $\ct{E}/X$, for every object $X$ in \ct{E}.
\item $(W, {\rm sup})$ is an indexed W-type for $f$.
\end{enumerate}
\end{theo}
\begin{proof} First we establish the equivalence of (1) and (2).

(1) $\Rightarrow$ (2): These properties are enjoyed by all initial algebras, so also by W-types.

(2) $\Rightarrow$ (1): Assume ${\rm sup}$ is an isomorphism and $W$ has no proper $P_f$-subalgebras. The latter means that we can prove properties of $W$ by induction. For if $L \subseteq W$ and $L$ is \emph{inductive} in the sense that
\[ \forall b \in B_a: tb \in L \Rightarrow {\rm sup}_a(t) \in L, \]
then $L$ defines a $P_f$-subalgebra of $W$ and therefore $L = W$.

Our first aim is to define a map
\[ \func{\rm tc}{W}{\spower W}, \]
that intuitively sends a tree to its transitive closure: the collection of all its subtrees, together with the tree itself. This we can do as follows. Call $A \in \spower W$ \emph{transitive}, when it is closed under subtrees. Formally: 
\[ {\rm sup}_a(t) \in A, b \in B_a \Rightarrow tb \in A. \]
Define TC($w$,$A$) to mean: $A$ is the least transitive subset of $W$ containing $a$. Formally:
\[ w \in A \land \forall B ( B \mbox{ is transitive} \land w \in B \Rightarrow A \subseteq B). \]
We can then define $L = \{ w \in W \, : \, \exists! A \in \spower W \, {\rm TC}(w, A) \}$. As $L$ is inductive, the object TC will be the graph of a function \func{\rm tc}{W}{\spower W}.

Now let $(X, \func{m}{P_{f}(X)}{X})$ be an arbitrary $P_{f}$-algebra. We need to construct a $P_f$-algebra morphism \func{k}{W}{X}. Intuitively, we do this by glueing together partial solutions to this problem, so-called attempts. An \emph{attempt} for an element $w \in W$ is a morphism \func{g}{{\rm tc(w)}}{X} with the property that for any tree ${\rm sup}_a(t) \in {\rm tc}(w)$ the following equality holds:
\[ g({\rm sup}_a t) = m(\lambda b \in B_a. g(tb)). \]
Intuitively, it is a $P_f$-algebra morphism \func{k}{W}{X} defined only on the transitive closure of $w$. Notice that there is an object of attempts in \ct{E}, because ${\rm tc}(w)$ is a small object for every $w \in W$, and the validity of {\bf (PE)} implies that of {\bf ($\Pi$E)}.

Our next aim is to show that for every $w \in W$ there is a unique attempt. Let $L$ be the collection of all those $w \in W$ such that for every $v \in {\rm tc}(w)$ there exists a unique attempt. We show that $L$ is inductive. So assume that for a fixed \func{t}{B_a}{W}, unique attempts $g_b$ have been defined for every $tb$ with $b \in B_a$. Now define an attempt for ${\rm sup}_a(t)$ by putting
\begin{displaymath}
\begin{array}{lcll}
g(v) & = & g_b(v) & \mbox{if } v \in {\rm tc}(tb), \\
g({\rm sup}_a(t)) & = & m(\lambda b \in B_a. g_b(tb)). &
\end{array}
\end{displaymath}
One readily sees that $g$ is the unique attempt for ${\rm sup}_a(t)$, so that ${\rm sup}_a(t)$ belongs to $L$. Therefore $L$ is inductive and unique attempts exist for every $w \in W$.

The desired map \func{k}{W}{X} can be defined by
\begin{eqnarray*}
k(w) = x & \Leftrightarrow & g(w) = x,
\end{eqnarray*}
where $g$ is the unique attempt for $w$. One uses the definition of an attempt to verify that this a $P_f$-algebra morphism. And it is the unique such, because restricting a $P_f$-algebra morphism $k$ to the transitive closure of a fixed tree $w$ gives an attempt for $w$.

(2) $\Rightarrow$ (3): If $T$ is a $P_{X^{*}f}$-subalgebra of $X^{*}W$ in $\ct{E}/X$, then
\[ L = \{ \, w \in W \, : \, \forall x \in X \, (x, w) \in T \, \} \]
defines a $P_f$-subalgebra of $W$ in \ct{E}. So if $W$ has no proper $P_f$-subalgebras, $X^*W$ has no proper $P_{X^*f}$-subalgebras.

(3) $\Rightarrow$ (4): This is the argument from (2) to (1) applied in all slices of \ct{E}.

(4) $\Rightarrow$ (1): By definition.
\end{proof}

We will need the notion of a collection span.
\begin{defi}{collspan}
A span $(g, q)$ in \ct{E}
\diag{ A & B \ar[r]^q \ar[l]_g & Y}
is called a \emph{collection span}, when, in the internal logic, it holds that for any map \func{f}{F}{B_a} covering some fibre of $g$, there is an element $a' \in A$ together with a cover \func{p}{B_{a'}}{B_a} over $Y$ which factors through $f$.

Diagrammatically, we can express this by asking that for any map $E \rTo A$ and any cover $F \rTo E \times _A B$ there is a diagram of the form
\diag{
& & Y & & \\
B \ar[d]_g \ar@/^/[urr]^q & E' \times_A B \ar[r] \ar[l] \ar[d] & F \ar@{->>}[r] & E \times_A B \ar[r] \ar[d] & B \ar[d]^g \ar@/_/[ull]_q \\
A & E' \ar@{->>}[rr] \ar[l] & & E \ar[r] & A,}
where the middle square is a covering square, involving the given map $F \rTo E \times _A B$, while the other two squares are pullbacks.
\end{defi}

\begin{lemm}{replbycollmap}
Assume \ct{E} is a category equipped with a representable class of display maps \smallmap{S} satisfying {\bf ($\Pi$E)}. Then every $\func{f}{Y}{X} \in \smallmap{S}$ fits into a covering square
\diag{ B \ar[d]_g \ar@{->>}[r]^q & Y \ar[d]^f \\
A \ar@{->>}[r]_p & X, }
where $g$ belongs to \smallmap{S} and $(g, q)$ is a collection span over $X$.
\end{lemm}
\begin{proof} As usual, we denote the representation of \smallmap{S} by $\pi: E \to U$.

We define $A$ by
\[ A = \Sigma_{x \in X, u \in U} \{ h:E_u \rTo Y_x \, : \, h \mbox{ is a cover} \},\]
and $p$ is the obvious projection. The fibre of $g$ above an element $(x, u, h)$ is $E_u$, and $q$ sends a pair $(x, u, h, e)$ with $(x, u, h) \in A$ and $e \in E_u$ to $h(e)$. It follows that $p$ is a cover, because $\pi$ is a representation, and the square is covering, because we require $h$ to be a cover.

When a cover \func{s}{T}{B_a} has been given for some $a = (x, u, h)$, there is an element $v \in U$ and a cover \func{t}{E_{v}}{B_a} factoring through $s$. (This is by using the Collection axiom {\bf (A7)} and representability.) Consider the element $a' = (x, v, qt) \in A$. The map
\diag{ B_{a'} \ar[r]^{\cong} & E_v \ar@{->>}[r]^t & B_a }
is a cover over $Y$ which factors through $s$.
\end{proof}

\begin{prop}{redlemmaWtypes}
Assume \ct{E} is an exact category with a class of small maps \smallmap{S} satisfying {\bf (PE)}. Assume furthermore $f \in \smallmap{S}$ fits into a covering square
\diag{ B \ar[d]_g \ar@{->>}[r]^q & Y \ar[d]^f \\
A \ar@{->>}[r]_p & X, }
where $(g, q)$ is a collection span over $X$, and $g$ is a small map for which the W-type exists. Then the W-type for $f$ also exists.
\end{prop}
\begin{proof}
Write $W$ for the W-type for $g$ and sup for the structure map. The idea is to use the well-founded trees in $W$, whose branching type is determined by $g$, to represent well-founded trees whose branching type is determined by $f$. In fact, $W_f$ will be obtained as a subquotient of $W$.

We wish to construct a binary relation $\sim$ on $W$ with the following property:
\begin{labequation}{despropofsim}
\begin{array}{lcl}
{\rm sup}_a t \sim {\rm sup}_{a'} t' & \Leftrightarrow & pa=pa' \mbox{ and } \\
& &  \forall b \in B_a, b' \in B_{a'} \, qb=qb' \Rightarrow tb \sim t'b'.
\end{array}
\end{labequation}%
We will call a relation $\sim$ with this property a \emph{bisimulation}, and using the inductive properties of $W$ we can prove that bisimulations on $W$ are unique. To see that there exists a bisimulation on $W$ we employ the same techniques as in \reftheo{charthmWtypes}. Recall in particular from the proof of \reftheo{charthmWtypes} the construction of a transitive closure ${\rm tc}(w)$ of an element $w \in W$: it is really the small object of all its subtrees, together with $w$ itself. In the same way, we can also define ${\rm st}(w)$, the collection of all subtrees of $w$ (not including $w$).

Since all diagonals are assumed to be small and $\spower 1$ classifies bounded subobjects, there is, for every object $X$, a function $X \times X \rTo \spower 1$ which assigns to every pair $(x, y) \in X \times X$ the small truth-value of the statement ``$x = y$''. We will denote it by $[ - = - ]$.
For a pair $(w,w') \in W^2$, call a function \func{g}{{\rm tc}(w) \times {\rm tc}(w')}{\spower 1} a \emph{bisimulation test}, when for all ${\rm sup}_a t \in {\rm tc}(w), {\rm sup}_{a'} t' \in {\rm tc}(w')$ the equality
\[ g({\rm sup}_a t, {\rm sup}_{a'} t') = [pa = pa'] \land \bigwedge_{b \in B_a, b' \in B_{a'}} ([qb = qb'] \rightarrow g(tb,t'b') ) \]
holds. Intuitively, a bisimulation test measures the degree to which two elements are bisimular, by sending a pair $(w,w')$ to the truth-value of the statement ``$w$ and $w'$ are bisimilar''. 

Our first aim is to show that for every pair $(w, w')$ there is a \emph{unique} bisimulation test. For this purpose, it suffices to show that for
\begin{eqnarray*}
L & = & \{ \, w \in W \, : \, \mbox{there is a unique bisimulation test} \\
& & \mbox{ for every pair } (w, w') \mbox{ with } w' \in W \, \}
\end{eqnarray*}
the following property holds:
\[ {\rm st}(w) \subseteq L \Rightarrow w \in L. \]
Because this will imply that $M = \{ w \in W \, : \, {\rm tc}(w) \subseteq L \}$ is inductive (i.e., defines a $P_{g}$-subalgebra of $W$), and therefore equal to $W$. As $M \subseteq L \subseteq W$ also $L = W$, and it follows that for every pair there is a unique bisimulation test. 

So let $w , w' \in W$ be given such that ${\rm st}(w) \subseteq L$. We need to show that for $(w, w')$ there is a unique bisimulation test $g$. We define $g(v, v')$ for $v \in {\rm tc}(w), v' \in {\rm tc}(w')$ as follows:
\begin{itemize}
\item If $v \in {\rm st}(w)$, then $v \in L$ and the pair $(v, v')$ has a unique bisimulation test $h$. We set $g(v, v') = h(v, v')$.
\item If $v = w = {\rm sup}_a t$ and $v' = {\rm sup}_{a'} t'$, then for every $b \in B_a$ and $b' \in B_{a'}$ we know that $tb \in {\rm st}(w) \subseteq L$ by induction hypothesis, and therefore there exists a unique bisimulation test $h_{b, b'}$ for $(tb, t'b')$. We set 
\[ g(v, v') := [pa = pa'] \land \bigwedge_{b \in B_a, b' \in B_{a'}} ([qb = qb'] \rightarrow h_{b, b'}(tb, t'b') ). \]
\end{itemize}
We leave to the reader the verification that this defines the unique bisimulation test $g$ for $(w, w')$.

Now we have established that for every pair there exists a unique bisimulation test, we can define a binary relation $\sim$ on $W$ by
\begin{eqnarray*}
w \sim w' & \Leftrightarrow & g(w,w') = \top,
\end{eqnarray*}
where $g$ is the unique bisimulation test for $(w, w')$. By construction, the relation $\sim$ is a \emph{bounded} bisimulation on $W$.

We can now show, using that $\sim$ is the unique bismimulation, that the relation is both symmetric and transitive. Since $\sim$ is bounded, it defines a bounded equivalence relation on the object $R = \{ w \in W \, : \, w \sim w \}$ of reflexive elements. Using bounded exactness, we can take its quotient $V$, writing $[-]$ for the quotient map $R \rTo V$. 

We claim $V$ is the W-type associated to $f$. To show that $V$ has the structure of a $P_f$-algebra, we need to define a map \func{s}{P_fV}{V}. So start with an $x \in X$ and a map \func{k}{Y_x}{V}. Choosing $a \in A$ to be such that $pa = x$, we have
\[ \forall b \in B_a \, \exists r \in R \, kqb = [r]. \]
Since $(g, q)$ is a collection span over $X$, there is a (potentially) different $a' \in A$ with $pa' = x$, and a map \func{t}{B_{a'}}{R} such that for all $b' \in B_{a'}$:
\[ kqb' = [tb']. \]
We set $s(x, k) = [{\rm sup}_{a'} t]$. The equivalence in \refeq{despropofsim} ensures that this value is independent of the choices we have made.

Finally, we use \reftheo{charthmWtypes} to prove that $(V, s)$ is the W-type for $f$. For showing that $s$ is an isomorphism, we need to construct an inverse $i$ for $s$. Now, every $v \in V$ is of the form $[w]$ for a reflexive element $w = {\rm sup}_a t$. Since $w$ is reflexive the equation
\[ k([b]) = [t(b)] \mbox{ for all } b \in B_a \]
defines a function \func{k}{Y_{pa}}{V}. So one may set $iv = (pa, k)$, which is, again by \refeq{despropofsim}, independent of the choice of $a$.

It remains to be shown that $V$ has no proper $P_f$-subalgebras. For this one proves that if $L$ is $P_f$-subalgebra of $V$, then
\[ T = \{ w \in W \, : \, w \sim w \Rightarrow [w] \in L \} \]
defines a $P_q$-subalgebra of $W$.
\end{proof}

\begin{lemm}{ypresexpons}
Let \func{{\bf y}}{(\ct{E}, \smallmap{S})}{(\overline{\ct{E}}, \overline{\smallmap{S}})} be the exact completion of a category with a class of display maps. Then $\bf y$ preserves the exponentials that exist in \ct{E}.
\end{lemm}
\begin{proof}
A trivial diagram chase: the key fact is that any object in $\overline{\ct{E}}$ arises as a quotient of an equivalence relation in \ct{E}.
\end{proof}

\begin{theo}{transfWE}
Let \func{{\bf y}}{(\ct{E}, \smallmap{S})}{(\overline{\ct{E}}, \overline{\smallmap{S}})} be the exact completion of a category with a representable class of display maps \smallmap{S} satisfying {\bf ($\Pi$E)} and {\bf (WE)}. Then $\overline{\smallmap{S}}$ satisfies {\bf (WE)} as well.
\end{theo}
\begin{proof} In this proof it might be confusing to drop the occurences of {\bf y}, so for once we insert them.

We first want to show that every map of the form ${\bf y}f$ with $f \in \smallmap{S}$ has a W-type in $\overline{\ct{E}}$. From \reflemm{ypresexpons} we learn that the functor ${\bf y}$ commutes with $P_f$. This means that ${\bf y}$ does also commute with $W$: using \reftheo{charthmWtypes}, we see that we only need to show that ${\bf y}W_f$ has no proper $P_{{\bf y}f}$-subalgebras. But this is immediate, since {\bf y} induces a bijective correspondence between Sub($W_f$) in \ct{E} and Sub(${\bf y}W_f$) in $\overline{\ct{E}}$.

Now the general case: by definition, any map $f \in \overline{\smallmap{S}}$ fits into a covering square as follows:
\diag{ {\bf y}X \ar[d]_{{\bf y}f'} \ar[r]^p &  A \ar[d]^f \\
{\bf y}Y \ar[r] & B,}
with $f' \in \smallmap{S}$. By \reflemm{replbycollmap}, $f'$ fits into a covering square in \ct{E}
\diag{ M \ar[d]_{g} \ar[r]^q &  X \ar[d]^{f'} \\
N \ar[r] & Y,}
where $g \in \smallmap{S}$ and $(g, q)$ is a collection span over $Y$. All of this is preserved by {\bf y}. Moreover, $({\bf y}g, p{\bf y}q)$ is a collection span over $B$. This means that we can apply \refprop{redlemmaWtypes} to deduce that a W-type for $f$ exists.
\end{proof}

\begin{coro}{WEundercov}
Let $(\ct{E}, \smallmap{S})$ be an exact category with a representable class of display maps \smallmap{S} satisfying {\bf($\Pi$E)}. When \smallmap{S} satisfies {\bf(WE)}, then so does $\smallmap{S}^{\rm cov}$.
\end{coro}

Again, we doubt whether a similar result for {\bf (WS)} holds.

\subsection{Infinity}

The following proposition is a triviality:
\begin{prop}{presnnocov}
Let $(\ct{E}, \smallmap{S})$ be a category with display maps \smallmap{S}. When \smallmap{S} satisfies {\bf (NE)} or {\bf (NS)}, then so does $\smallmap{S}^{\rm cov}$. 
\end{prop}
The following, however, less so:
\begin{prop}{presnno}
Let \func{{\bf y}}{(\ct{E}, \smallmap{S})}{(\overline{\ct{E}}, \overline{\smallmap{S}})} be the exact completion of a category \ct{E} with a representable class of display maps \smallmap{S} satisfying {\bf ($\Pi$E)}. When \smallmap{S} satisfies {\bf (NE)} or {\bf (NS)}, then so does $\overline{\smallmap{S}}$.
\end{prop}
\begin{proof}
The statement follows immediately from the fact that {\bf y} preserves the natural numbers object $\NN$, whenever it exists in \ct{E}. But as $\NN$ is the W-type associated to the left sum inclusion \func{i}{1}{1 + 1}, this can be shown as in \reftheo{transfWE}: we only need to show that ${\bf y}\NN$ has no proper $P_{i}$-subalgebras (by \reftheo{charthmWtypes}), which follows from the fact that {\bf y} is bijective on subobjects.
\end{proof}

\subsection{Fullness}

In this subsection we discuss the stability properties of the Fullness axiom, which are rather good. To show this, we first prove two lemmas, the second of which is also useful in other contexts.

\begin{lemm}{usefulforfull}
Let $(\ct{E}, \smallmap{S})$ be a category with display maps. Any composable pair of arrows in $\ct{E}$ of the form
\diag{ C \ar@{ >->}[r]^m & B \ar[r]^f & A}
with $m \in \smallmap{S}^{\rm cov}$ a mono and $f \in \smallmap{S}$, fits into a diagram of the form
\diag{ Z \ar@{ >->}[d]_n \ar@{->>}[r] & C \ar@{ >->}[d]^m \\
Y \ar[d]_g \ar@{->>}[r] & B \ar[d]^f \\
X \ar@{->>}[r]_p & A,}
where both squares are pullbacks, the horizontal arrows are covers (as indicated), and both $n$ and $g$ belong to \smallmap{S}.
\end{lemm}
\begin{proof}
Using the definition of $\smallmap{S}^{\rm cov}$, we know that $m$ is covered by a map in \smallmap{S}. Using axiom {\bf (A10)}, we may actually assume that $m$ is covered via a pullback square by a mono $m' \in \smallmap{S}$. Then using Collection for \smallmap{S}, we obtain a diagram of the form
\diag{ Z' \ar@{ >->}[d]_{n'} \ar[r] & C' \ar@{ >->}[d]^{m'} \ar@{->>}[r] & C \ar@{ >->}[d]^m \\
Y' \ar[d]_{g'} \ar[r] & B' \ar@{->>}[r] & B \ar[d]^f \\
X \ar@{->>}[rr]_p & & A,}
where the top squares are pullbacks and the rectangle below is covering, and both $n'$ and $g'$ belong to \smallmap{S}. By pulling back $m$ and $f$ along $p$, we obtain a diagram as follows:
\diag{ Z' \ar@{ >->}[d]_{n'} \ar@{->>}[r] & Z \ar@{ >->}[d]^{p^*m=n} \ar@{->>}[r] & C \ar@{ >->}[d]^m \\
Y' \ar[dr]_{g'} \ar@{->>}[r]^q & Y \ar@{->>}[r] \ar[d]^{p^*f = g} & B \ar[d]^f \\
& X \ar@{->>}[r]_p & A,}
where the squares are all pullbacks. Then $g \in \smallmap{S}$ by pullback stability, and $q \in \smallmap{S}$ by local fullness (or \reflemm{composlemma}). Since $n' \in \smallmap{S}$, also $qn' \in \smallmap{S}$ by closure under composition. Then $n \in \smallmap{S}$ by axiom {\bf (A10)}.
\end{proof}

\begin{lemm}{redlemmforfull}
Let $(\ct{E}, \smallmap{S})$ be a category with a class of display maps. Suppose we are given in \ct{E} a diagram of the form
\diag{ B_0 \ar@{->>}[r] \ar[d]_{\psi} & B \ar[d]^{\phi} \\
A_0 \ar@{->>}[r] \ar[d]_i & A \ar[d]^j \\
X_0 \ar@{->>}[r]_p & X, }
in which both squares are covering and $\psi$ and $i$ belong to \smallmap{S} and $\phi$ and $j$ belong to $\smallmap{S}^{\rm cov}$. If a generic $\smallmap{S}$-displayed \emph{mvs} for $\psi$ exists, then also a generic $\smallmap{S}^{\rm cov}$-displayed \emph{mvs} for $\phi$ exists.
\end{lemm}
\begin{proof}
By pulling back $\phi$ along $p$, we obtain over $X_0$ the following covering square:
\diag{ B_0 \ar@{->>}[r]^{\delta} \ar[d]_{\psi} & p^* B \ar[d]^{p^*\phi} \\
A_0 \ar@{->>}[r] &  p^* A. }
By \reflemm{composlemma}, all arrows in this square belong to $\smallmap{S}^{\rm cov}$.

Using Fullness for $\psi$, we find a cover \func{e}{X'}{X_0} and a map $\func{s}{Y}{X'} \in \smallmap{S}$, together with a generic \smallmap{S}-displayed \emph{mvs} $P$ for $\psi$ over $Y$. Writing $\kappa$ for the composite $\func{es}{Y}{X_0}$ and $\alpha$ for \func{p \kappa}{Y}{X}, we obtain the following covering square over $Y$:
\diag{ \kappa^* B_0 \ar@{->>}[r]^{\kappa^*\delta} \ar[d]_{ \kappa^* \psi} & \alpha^* B \ar[d]^{\alpha^* \phi} \\
\kappa^* A_0 \ar@{->>}[r] & \alpha^* A. }
All the arrows in this square belong to $\smallmap{S}^{\rm cov}$, so the \smallmap{S}-displayed \emph{mvs} $P$ of $\psi$ over $Y$ induces a $\smallmap{S}^{\rm cov}$-displayed \emph{mvs} $\overline{P}$ of $\phi$ over $Y$ by $\overline{P} = (\kappa^* \delta)_* P$. We claim it is generic.

So let $\func{t}{Z}{X'}$ be any map and $\overline{Q}$ be an $\smallmap{S}^{\rm cov}$-displayed \emph{mvs} of $\phi$ over $Z$. Writing $\lambda = et$ and $\beta = p \lambda$, we obtain a diagram over $Z$ as follows:
\diag{ Q' \ar@{ >->}[d] \ar[r] & \overline{Q} \ar@{>->}[d] \\
\lambda^* B_0 \ar[r]^{\lambda^*\delta} \ar[d]_{\lambda^*\psi} & \beta^* B \ar[d]^{\beta^* \phi} \\
\lambda^* A_0 \ar[r]  & \beta^* A, }
with $Q' = (\lambda^* \delta)^* \overline{Q}$. Because all arrows in this diagram belong to $\smallmap{S}^{\rm cov}$, and $\overline{Q}$ is an $\smallmap{S}^{\rm cov}$-displayed \emph{mvs} for $\phi$ over $Z$, the subobject $Q'$ is an $\smallmap{S}^{\rm cov}$-displayed \emph{mvs} for $\psi$ over $Z$.

Notice that we have obtained a diagram of the form
\diag{ Q' \ar@{ >->}[r] & \lambda^* B \ar[rr]^{\lambda^*(i\psi)} &  & Z, }
where the first map belongs to $\smallmap{S}^{\rm cov}$ and the second belongs to \smallmap{S}. So we can use the previous lemma to obtain a cover \func{v}{Z'}{Z} such that $Q = v^* Q'$ is an \smallmap{S}-displayed \emph{mvs} of $\psi$ over $Z'$.

By genericity of $P$, this means that we find a map \func{y}{U}{Y} and a cover \func{q}{U}{Z'} with $sy = tvq$ such that $y^* P \leq q^* Q$ as displayed \emph{mvs}s of $\psi$ over $U$. Now 
\[ \kappa y = esy = etvq = \lambda v q, \]
and therefore also 
\[ ((\kappa y)^* \delta)_* y^*P \leq ((\lambda v q)^* \delta)_* q^* Q = ((\lambda v q)^* \delta)_* (vq)^* Q' \]
as displayed \emph{mvs}s of $\phi$ over $U$. But 
\[ ((\kappa y)^* \delta)_* y^*P = y^* (\kappa^* \delta)_* P = y^* \overline{P}, \]
and 
\[ ((\lambda v q)^* \delta)_* (vq)^* Q' = (vq)^* (\lambda^* \delta)_* Q' = (vq)^* (\lambda^* \delta)_* (\lambda^* \delta)^* \overline{Q} \leq (vq)^* \overline{Q}. \]
This completes the proof.
\end{proof}
\begin{prop}{presFcov}
Let $(\ct{E}, \smallmap{S})$ be a category with display maps \smallmap{S}. When \smallmap{S} satisfies {\bf (F)}, then so does $\smallmap{S}^{\rm cov}$. 
\end{prop}
\begin{proof}
Immediate from the previous lemma using \reflemm{1stuseofcoll}.
\end{proof}

Now the proof of the main result of this subsection should be straightforward:
\begin{prop}{presofFullness}
Let \func{{\bf y}}{(\ct{E}, \smallmap{S})}{(\overline{\ct{E}}, \overline{\smallmap{S}})} be the exact completion of a category with a class of display maps \smallmap{S}. When \smallmap{S} satisfies {\bf (F)}, then so does $\overline{\smallmap{S}}$.
\end{prop}
\begin{proof}
Once again, we systemically suppress occurences of ${\bf y}$.

In view of \reflemm{2nduseofcoll} and \reflemm{redlemmforfull}, it suffices to show that a generic $\overline{\smallmap{S}}$-displayed \emph{mvs} exists in $\overline{\ct{E}}$ for those $\func{\phi}{B}{A} \in \smallmap{S}$ with $A \rTo X \in \smallmap{S}$. Of course, because Fullness holds for $\phi$ in \ct{E}, there is a cover \func{e}{X'}{X} and a map $\func{s}{Y}{X'} \in \smallmap{S}$, together with an \smallmap{S}-displayed \emph{mvs} $P$ for $\phi$ which is generic in \ct{E}. We claim it is also a generic $\overline{\smallmap{S}}$-displayed \emph{mvs} for $\phi$ in $\overline{\ct{E}}$.

So let $\func{t}{Z}{X'}$ be any map and $Q$ be an $\overline{\smallmap{S}}$-displayed \emph{mvs} of $\phi$ over $Z$. As ${\bf y}$ is covering, we obtain a cover \func{p}{Z_0}{Z} with $Z_0 \in \ct{E}$. Writing $\lambda = etq$, we obtain the following diagram in \ct{E} (!):
\diag{ p^*Q \ar@{ >->}[r] & \lambda^* B \ar[rr]^{\lambda^*(i\phi)} &  & Z_0, }
where the first arrow belongs to $\smallmap{S}^{\rm cov}$ and the second arrow to \smallmap{S}. Then, using \reflemm{usefulforfull}, we find a cover \func{q}{Z_1}{Z_0} in \ct{E} such that $(pq)^* Q$ is an \smallmap{S}-displayed \emph{mvs} for $\phi$ over $Z_1$.

Using the genericity of $P$, this means there exist a map \func{y}{U}{Y} and a cover \func{r}{U}{Z_1} with $sy = tpqr$ such that $y^* P \leq r^* (pq)^* Q = (pqr)^* Q$ as \smallmap{S}-displayed, and therefore also $\overline{\smallmap{S}}$-displayed, \emph{mvs}s of $\phi$ over $U$. This completes the proof.
\end{proof}

\part*{A categorical semantics for set theory}

In this final part of the paper we will explain how categories with small maps provide a semantics for set theory. In Section 7, we establish its soundness, and in Section 8 its completeness.

\section{Soundness}

Throughout this section $(\ct{E}, \smallmap{S})$ will be a bounded exact category with a representable class of small maps \smallmap{S} satisfying {\bf ($\Pi$E)} and {\bf (WE)}. We will refer to this as a \emph{predicative category with small maps}.\footnote{Compare the notion of a $\Pi W$-pretopos or a ``predicative topos'' in \cite{moerdijkpalmgren00} and \cite{berg05}.} 

It follows from \refcoro{PiEimpliesPE} that {\bf (PE)} holds in \ct{E} as well, so it makes sense to consider (indexed) \spower-algebras in \ct{E} (see \refdefi{initalg} and \reflemm{PEstableunderslicing}). In particular, it makes sense to ask whether the indexed initial \spower-algebra exists in \ct{E}. For the moment we will simply assume that it does and denote it by $V$.

Since $V$, as an initial algebra, is a fixed point of \spower, it comes equipped with with two mutually inverse maps:
\diag{ \spower V \ar@/^/[rr]^{\rm Int} & & V \ar@/^/[ll]^{\rm Ext}. }
In the internal logic of \ct{E}, we can therefore define a binary relation $\epsilon$ on $V$, as follows:
\[ x \epsilon y \Leftrightarrow x \in {\rm Ext} (y). \]
In this way, we obtain a structure $(V, \epsilon)$ in the language of set theory, and the next result shows that it models a rudimentary set theory {\bf RST} (see Appendix A for its axioms).
\begin{prop}{regcompl}
Assume the indexed initial \spower-algebra $V$ exists, and $\epsilon$ is the binary predicate defined on it as above. Then all axioms of {\bf RST} are satisfied in the structure $(V, \epsilon)$. 
\end{prop}
\begin{proof} By the universal property of power objects, there is a correspondence between small subobjects $A \subseteq V$ (i.e., subobjects $A$ of $V$ such that $A \rTo 1$ is small) and elements of $\spower (V)$. Therefore we can call $y \in V$ the \emph{name} of the small subobject $A \subseteq V$, in case ${\rm Ext}(y)$ is the corresponding element in $\spower(V)$. 

We verify the validity of the axioms of {\bf RST} by making extensive use of the internal language of the positive Heyting category \ct{E}. 

Extensionality holds because two small subobjects ${\rm Ext}(x)$ and ${\rm Ext}(y)$ of $V$ are equal if and only if, in the internal language of \ct{E},   $z\in {\rm }{\rm Ext}(x)\leftrightarrow z\in {\rm Ext}(y)$. The least subobject $0\subseteq V$ is small, and its name \func{\emptyset}{1}{V} models the empty set. The pairing of two elements $x$ and $y$ represented by two arrows $1\rTo V$, is given by ${\rm Int}(l)$, where $l$ is the name of the (small) image of their copairing \func{[x,y]}{1+1}{V}. The union of the sets contained in a set $x$ is interpreted by applying the multiplication of the monad $\spower$ to $(\spower {\rm Ext})({\rm Ext}(x))$:
\diag{ {\rm Ext}(x) \in \spower V \ar[rr]^{\spower {\rm Ext}} & & \spower \spower V \ar[r]^{\mu_V} & \spower V \ar[r]^{\rm Int} & V.} 

To show the validity of Bounded separation, we need to observe that $=$ and $\epsilon$ are bounded relations on $V$. So for any bounded formula $\phi$ in the language of set theory and $a \in V$, the subobject $S$ of ${\rm Ext}(a)$ defined by
\[ S = \{ y \in {\rm Ext}(a) \, : \, V \models \phi(y) \} \]
is bounded, and hence small. The name $x$ of $S$ now satisfies $\ \forall y \, ( \, y \epsilon x \leftrightarrow y \epsilon a \land \phi(y) \, ) $.

To show the validity of Strong collection, assume $\forall x \epsilon a \exists y \phi(x, y)$ holds. Then we have a cover $p_1$
\diag{ E = \{ (x, y) \in V^2 \, : \, \, V \models \phi(x, y) \land x \epsilon a \} \ar@{ ->>}[r] & {\rm Ext}(a), }
given by the first projection. Since ${\rm Ext}(a)$ is small, there is a small object $S$ together with a cover \func{q}{S}{{\rm Ext}(a)} factoring through $p_1$. So there is a map \func{f}{S}{E} with $p_1f = q$. Consider the image of \func{p_2f}{S}{V}, where $p_2$ is the second projection: its name $b$ provides the right bounding set to witness the desired instance of the Strong collection scheme.

So far we have only used that $V$ is a fixed point, but to verify Set induction we use that it is indexed initial as well. If $\forall y  \epsilon x \, \phi(y) \rightarrow \phi(x)$ holds in $V$, then $L = \{ x \in V \, : \, V \models \phi(x) \}$ is a \spower-subalgebra of $V$. But the initial \spower-algebra has no proper \spower-subalgebras, so then $L \cong V$ and $\forall x \, \phi(x)$ holds in $V$.
\end{proof}

Several questions arise: is it possible to extend this result to cover the set theories {\bf IZF} and {\bf CZF}? The next proposition shows the answer to this question is \emph{yes}. Another question would be: does the indexed initial \spower-algebra always exist? As it turns out, the answer to this question is affirmative as well.

\begin{prop}{moreaxioms}
Assume the indexed initial \spower-algebra $V$ exists, and $\epsilon$ is the binary predicate defined on it as above. 
\begin{enumerate}
\item When \smallmap{S} satisfies {\bf (M)}, then $(V, \epsilon)$ validates the Full separation scheme.
\item When \smallmap{S} satisfies {\bf (PS)}, then $(V, \epsilon)$ validates the Power set axiom.
\item When \smallmap{S} satisfies {\bf (NS)}, then $(V, \epsilon)$ validates the Infinity axiom.
\item When \smallmap{S} satisfies {\bf (F)}, then $V$ validates the Fullness  axiom.
\end{enumerate}
\end{prop}
\begin{proof} We again make extensive use of the internal language of \ct{E}.
\begin{enumerate}
\item The argument for Full separation is identical to the one for bounded Separation.
\item When \smallmap{S} satisfies {\bf (PS)}, then $\spower({\rm Ext}(x))$ is small for any $x \in V$. The same applies to the image of
\diag{ \spower({\rm Ext}(x)) \ar@{ >->}[r] & \spower V \ar[r]^{\rm Int} & V, }
whose name $y$ will be the small power set of $x$.
\item The morphism \func{\emptyset}{1}{V}, together with the map \func{s}{V}{V} which takes an element $x$ to $x\cup\{x\}$, yields a morphism \func{\alpha}{\NN}{V}. When $\NN$ is small, so is the image of $\alpha$, as a subobject of $V$. Applying ${\rm Int}$ to its name we get an infinite set in $V$.
\item Assuming that \smallmap{S} satisfies {\bf (F)}, there is for any function $\func{f}{b}{A} \in V$ a small subobject $Z \in \spower {\rm Ext}(b)$ of multi-valued sections of \func{{\rm Ext}(f)}{{\rm Ext}(b)}{{\rm Ext}(a)} that is full (in the sense that any \emph{mvs} contains one in this set). The value $z$ of $Z$ under the map
\diag{ \spower({\rm Ext}(b)) \ar@{ >->}[r] & \spower V \ar[r]^{\rm Int} & V, }
is then a full set of \emph{mvs}s of $f$ in $V$.
\end{enumerate}
\end{proof}

We will now prove the existence of an initial \spower-algebra in \ct{E}. The proof makes essential use of the exactness of \ct{E}, and, as mentioned before, it is one of our reasons for insisting on exactness for predicative categories with small maps. The idea behind this result, which shows how initial \spower-algebras can be constructed in the presence of W-types, is essentially due to Aczel in \cite{aczel78}. The first application of this idea in a categorical context was in \cite{moerdijkpalmgren02}. But before we go into the proof of this result, we first borrow from \cite{kouwenhovenvanoosten05} the following characterisation theorem for initial \spower-algebras (compare \reftheo{charthmWtypes}).
\begin{theo}{charinitpsalg}
Let \ct{E} be a category with a class of small maps \smallmap{S} satisfying {\bf (PE)}. The following are equivalent for a $\spower$-algebra $(V, \func{\rm Int}{\spower(V)}{V})$:
\begin{enumerate}
\item $(V, {\rm Int})$ is the initial $\spower$-algebra.
\item The structure map ${\rm Int}$ is an isomorphism and $V$ has no proper $\spower$-subalgebras in $\ct{E}$.
\item The structure map ${\rm Int}$ is an isomorphism and $X^{*}V$ has no proper $\slspower{X}$-subalgebras in $\ct{E}/X$, for every object $X$ in \ct{E}.
\item $(V, {\rm Int})$ is the indexed initial $\spower$-algebra.
\end{enumerate}
\end{theo}
\begin{proof}
See \cite{kouwenhovenvanoosten05} and \reftheo{charthmWtypes}.
\end{proof}

Note that the characterisation theorem also shows that initial \spower-algebras are always indexed.

\begin{theo}{existinitpsalg}
If $(\ct{E}, \smallmap{S})$ is a predicative category with small maps, then the initial \spower-algebra exists in \ct{E}.
\end{theo}
\begin{proof}
The proof is very similar to that of \refprop{redlemmaWtypes}, so we will frequently refer to that proof for more details. In particular, we again construct a bisimulation on a W-type, which can be done by glueing together local solutions given by bisimulation tests.

Consider $W = W_{\pi}$, the W-type associated to the representation \func{\pi}{E}{U} for \smallmap{S}. To obtain the initial \spower-algebra, we want to quotient $W$ by \emph{bisimulation}, by which we now mean a binary relation $\sim$ on $W$ such that
\begin{eqnarray*}
{\rm sup}_u(t) \sim {\rm sup}_{u'}(t') & \Leftrightarrow & \forall e \in E_u \, \exists e' \in E_{u'} \, te \sim t'e' \\
& & \mbox{ and } \forall e' \in E_{u'} \, \exists e \in E_{u} \, te \sim t'e'.
\end{eqnarray*}
It can again be shown by induction that bisimulations are unique, but the difficulty is to show that they exist.

Using the notion of a transitive closure from the proof of \reftheo{charthmWtypes}, we define the appropriate notion of a \emph{bisimulation test}. For a pair $(w,w') \in W^2$, call a function \func{g}{{\rm tc}(w) \times {\rm tc}(w')}{\spower 1} a bisimulation test, when for all ${\rm sup}_u t \in {\rm tc}(w), {\rm sup}_{u'} t' \in {\rm tc}(w')$ the equality
\[ g({\rm sup}_u t, {\rm sup}_{u'} t') = \bigwedge_{e \in E_u} \bigvee_{e' \in E_{u'}} g(te, t'e') \wedge \bigwedge_{e' \in E_{u'}} \bigvee_{e \in E_{u}} g(te, t'e') \]
holds. In the manner of \refprop{redlemmaWtypes} it can be shown that there is a unique bisimulation test for every pair $(w,w')$. 

Using now that for every pair there is a unique bisimulation test, we can define the desired bisimulation $\sim$ by putting
\begin{eqnarray*}
w \sim w' & \Leftrightarrow & g(w,w') = \top,
\end{eqnarray*}
if $g$ is the unique bisimulation test for $(w, w')$. By construction it is a bisimulation, which is also bounded.

Using the inductive properties of $W$ again, we can see that any bisimulation on $W$ is an equivalence relation. So $\sim$ is a bounded equivalence relation for which we can take the quotient $V = W/ \sim$, with quotient map $q$. 

We claim $V$ is the initial \spower-algebra. We first need to see that it is a fixed point for the \spower-functor. To this end, we consider the solid arrows in the following diagram
\diag{ P_{\pi} W \ar@{->>}[r]^{\tau_W} \ar[d]_{\rm sup} & \spower W \ar@{->>}[r]^{\spower q} & \spower V \ar@/^/@{.>}[d]^{\rm Int} \\
W \ar@{->>}[rr]_q & & V, \ar@/^/@{.>}[u]^{\rm Ext} }
where $\tau_W$ is the component on $W$ of the natural transformation in \refcoro{PiEimpliesPE}, and $\spower q$ is a cover by \refprop{collandPE}.
One quickly sees that the notion of a bisimulation is precisely such that maps Int and Ext making the above diagram commute have to exist. To see that it is the initial \spower-algebra, we use the criterion in \reftheo{charinitpsalg}. Simply note that if $L$ is a (proper) \spower-subalgebra of $V$, then
\[ q^{-1}L = \{ w \in W \, : \, q(w) \in L \} \]
is a (proper) $P_{\pi}$-subalgebra of $W$.
\end{proof}

The next result summarises the results we have obtained in this section:
\begin{coro}{soundness}
Let $(\ct{E}, \smallmap{S})$ be a predicative category with small maps. Then $(\ct{E}, \smallmap{S})$ contains a model $(V, \epsilon)$ of the set theory {\bf RST} given by the initial \spower-algebra. Moreover, if \smallmap{S} satisfies the axioms {\bf (NS), (M)} and {\bf (PS)}, the structure $(V, \epsilon)$ models {\bf IZF}; and if the class of small maps satisfies {\bf (NS)} and {\bf (F)}, it is a model of {\bf CZF}.
\end{coro}

Completeness of this semantics for all three set theories will be proved in the next section.

\begin{rema}{compwithothsettings} This might be the right time to compare our approach and concepts to some of the other ones available in the literature, and our notion of a predicative category with small maps in particular.

In the book ``Algebraic Set Theory'' \cite{joyalmoerdijk95}, the basic notion of a category with small maps on pages 7--9 is given by a Heyting pretopos \ct{E} with an nno equipped with a representable class of maps satisfying {\bf (A1-7)} and {\bf ($\Pi$E)}. Our notion of a predicative category with small maps is both stronger and weaker: it is stronger, because we have added the axioms {\bf (A8)} and {\bf (A9)}, as well as {\bf (WE)}; it is weaker, because we have bounded exactness only. To be absolutely precise: in \cite{joyalmoerdijk95} the authors work with a different notion of representability, which is equivalent to ours in the bounded exact context (as was shown in \refprop{reprinJMsense}), but probably stronger in the context of Heyting pretoposes.

For showing the existence of the initial \spower-algebra in Chapter III of \cite{joyalmoerdijk95}, the authors make the additional assumption of the presence of a subobject classifier in \ct{E}. This assumption was too impredicative for our purposes, so therefore we have assumed the existence of W-types in the form of {\bf (WE)} instead. (Note that the existence of a subobject classifier, as well as the axioms {\bf (A8)} and {\bf (A9)}, all follow from the impredicative axiom {\bf (M)}.)\footnote{We suspect that {\bf (WE)} follows from the existence of a subobject classifier, but we haven't checked this.}

In \cite{awodeywarren05}, Awodey and Warren call a positive Heyting category with a class of maps satisfying {\bf (A1-9)} and {\bf (PE)}, with the possible exception of the Collection axiom {\bf (A7)}, a \emph{basic class structure}. To this, our notion of a predicative category with small maps adds the Collection axiom {\bf (A7)}, bounded exactness, representability and {\bf (WE)}. But note that all these axioms are valid in the ideal models that they study.
\end{rema}

\section{Completeness}

In this section we will show that the semantics for {\bf IZF} and {\bf CZF} we have developed in the previous section is complete. In order to show this, we need to make our ``informal example'' (cf.~\refrema{infexample}) more concrete. This we can do in two ways: either we can consider the classes and sets of {\bf IZF} and {\bf CZF} as being given by formulas from the language, or we work relative to a model. To be more precise:
\begin{rema}{syntcateg}
For any set theory {\bf T} extending {\bf RST} we can build the \emph{syntactic category} $\ct{E}[{\bf T}]$. Objects of this category are the ``definable classes'', meaning expressions of the form $\{ x : \phi(x) \}$, while identifying syntactic variants. Morphisms are ``definable class morphisms'': a morphism from the object $\{ x : \phi(x) \}$ to $\{ y : \psi(y) \}$, where we can assume that $x$ and $y$ are different, is an equivalence class of formulas $\alpha(x, y)$ such that the following is derivable in {\bf T}:
\[ \forall x \, ( \, \phi(x) \rightarrow \exists ! y \, ( \, \psi(y) \land \alpha(x, y) \, ) \, ). \]
Two such formulas $\alpha(x,y)$ and $\beta(x, y)$ are identified when {\bf T} proves
\[ \forall x \, \forall y \, ( \, \phi(x) \land \psi(y) \rightarrow ( \, \alpha(x, y) \leftrightarrow \beta(x, y) \, ) \, ). \]
One readily shows that this syntactic category is a positive Heyting category. It is actually a category with small maps, when, following the intuition, we declare those class morphisms whose fibres are sets to be small. So a morphism represented by $\alpha(x,y)$ from the object $\{ x : \phi(x) \}$ to $\{ y: \psi(y) \}$ is a small map, when {\bf T} proves
\[ \forall y \, (\, \psi(y) \rightarrow \exists a \, \forall x \, ( \, x \epsilon a \leftrightarrow \alpha(x, y) \land \phi(x) \, ) \, ). \]
The category with small maps obtained in this way will be denoted by $(\ct{E}[{\bf T}], \smallmap{S}[{\bf T}])$.
\end{rema}

\begin{rema}{catofsmmfrommodel}
Let $(M, \epsilon)$ be a structure (in the ordinary, set-theoretic sense) having the signature of the language of set theory, modelling the set-theoretic axioms of {\bf RST}. By the same construction as in the previous example, but replacing everywhere derivability in {\bf T} by validity in $M$, we obtain a category with small maps $(\ct{E}[M], \smallmap{S}[M])$ from $M$.
\end{rema}

The main results about these two examples are the following:
\begin{prop}{validinsyntcat}
For a set theory {\bf T} extending {\bf RST}, the class of small maps $\smallmap{S}[{\bf T}]$ in the syntactic category $\ct{E}[{\bf T}]$ is representable and satisfies {\bf (PE)} and {\bf (WE)}. Moreover, when 
\[ V[{\bf T}] = \{ x \, : \, x = x \} \]
is the class of all sets in $(\ct{E}[{\bf T}], \smallmap{S}[{\bf T}])$, then $V[{\bf T}]$ is the initial \spower-algebra, and for any set-theoretic sentence $\phi$:
\[ V[{\bf T}] \models \phi \Leftrightarrow {\bf T} \vdash \phi. \]
\end{prop}
\begin{proof} We first describe a representation \func{\pi}{E}{U} of $\smallmap{S}[{\bf T}]$. $U$ is the class of all sets $\{ x \, :\,  x=x \}$, while 
\begin{labequation}{membership}
E = \{ x \, : \, \exists y \, \exists z \, ( \, x = (y, z) \land y \epsilon z \, ) \}.
\end{labequation}%
The map $\pi$ is the projection on the second coordinate (here and below we are implicitly using some coding of $n$-tuples in set theory). 

The description of the $\spower$-functor on $(\ct{E}[{\bf T}], \smallmap{S}[{\bf T}])$ is what one would think it is. For an object $X = \{ x : \phi(x) \}$ of $\ct{E}[{\bf T}]$, the power class object $\spower (X)$ is given by
\[ \{ y \, : \, \forall x \epsilon y \, \phi(x) \}, \]
showing that {\bf (PE)} holds in the syntactic category. That $\smallmap{S}[{\bf T}]$ satisfies {\bf (WE)} as well follows from Example 3 on page 5--4 of \cite{aczelrathjen01}.

In view of the description of the $\spower$-functor above, it is clear that $V[{\bf T}]$ one of its fixed points. Since {\bf T} includes the Set induction scheme, $V[{\bf T}]$ is actually a fixed point for $\spower$ having no proper \spower-subalgebras. So it is the initial \spower-algebra by \reftheo{charinitpsalg}. 

It is also easy to see that the membership relation induced on $V[{\bf T}]$ is given by $E$ in \refeq{membership}. In general, one can prove by induction on its complexity that any set-theoretic formula $\phi(x_1, \ldots, x_n)$ is interpreted by the subobject of $V[{\bf T}]^n$ given by:
\[ \{ x \, : \, \exists x_1, \ldots, x_n \, ( \, x = (x_1, \ldots, x_n) \land \phi(x_1, \ldots, x_n)  \, ) \}. \]
From this and the definition of morphisms in $\ct{E}[{\bf T}]$ it follows that derivability in the set theory {\bf T} coincides with validity in the model $V[{\bf T}]$.
\end{proof}

\begin{prop}{validinmodcat}
Let $(M, \epsilon)$ be a structure (in the sense of model theory) modelling {\bf RST}. Then the class of small maps $\smallmap{S}[M]$ in the category $(\ct{E}[M], \smallmap{S}[M])$ is representable and satisfies {\bf (PE)} and {\bf (WE)}. Moreover, when 
\[ V[M] = \{ x \, : \, x = x \} \]
is the class of all sets in $(\ct{E}[{\bf T}], \smallmap{S}[{\bf T}])$, then $V[M]$ is the initial \spower-algebra, and for any set-theoretic sentence $\phi$:
\[ V[M] \models \phi \Leftrightarrow M \models \phi. \]
\end{prop}
\begin{proof} As in \refprop{validinsyntcat}.
\end{proof}

The last proposition makes clear how our categorical semantics extends the usual set-theoretic one.

We can now use the syntactic category to obtain a strong completeness result for {\bf RST}.
\begin{theo}{complforRST} For any set theory {\bf T} extending {\bf RST} there is a predicative category with small maps $(\ct{E}, \smallmap{S})$ such that for the initial \spower-algebra $V$ in $(\ct{E}, \smallmap{S})$ we have
\[ V \models \phi \Leftrightarrow {\bf T} \vdash \phi \]
for every set-theoretic sentence $\phi$. Therefore a set-theoretic sentence valid in every initial \spower-algebra in a predicative category with small maps $(\ct{E}, \smallmap{S})$ is a consequence of the axioms of {\bf RST}.
\end{theo}
\begin{proof}
For the predicative category with small maps $(\ct{E}, \smallmap{S})$ we take the exact completion of the syntactic category $(\ct{E}[{\bf T}], \smallmap{S}[{\bf T}])$ associated to ${\bf T}$.

We claim that the image ${\bf y}V[{\bf T}]$ of the initial \spower-algebra $V[{\bf T}]$ in the syntactic category is the initial \spower-algebra $V$ in \ct{E}. Since the embedding {\bf y} commutes with $\spower$ (by \refprop{presPE}), the object ${\bf y}V[{\bf T}]$ is still a fixed point for $\spower$ in \ct{E}. It does not have any proper \spower-subalgebras, because {\bf y} commutes with $\spower$ and is bijective on subobjects. Therefore it is the initial \spower-algebra $V$ by \reftheo{charinitpsalg}. 

Finally, the embedding {\bf y} is a Heyting functor which is bijective on subobjects, so we have for any set-theoretic sentence $\phi$:
\[ V \models \phi \Leftrightarrow V[{\bf T}] \models \phi \Leftrightarrow {\bf T} \vdash \phi. \]
\end{proof}

To extend this strong completeness result to {\bf IZF} and {\bf CZF} we need to prove the following proposition:
\begin{prop}{convforsyntcat}
Let $(\ct{E}[{\bf T}], \smallmap{S}[{\bf T}])$ be the syntactic category associated to a set theory ${\bf T}$ extending {\bf RST}. Then
\begin{enumerate}
\item ${\bf T} \vdash \mbox{\rm Full separation} \Rightarrow \smallmap{S}[{\bf T}]$ satisfies {\bf (M)}.
\item ${\bf T} \vdash \mbox{\rm Power set} \Rightarrow \smallmap{S}[{\bf T}]$ satisfies {\bf (PS)}.
\item ${\bf T} \vdash \mbox{\rm Infinity} \Rightarrow \smallmap{S}[{\bf T}]$ satisfies {\bf (NS)}.
\item ${\bf T} \vdash \mbox{\rm Fullness} \Rightarrow \smallmap{S}[{\bf T}]$ satisfies {\bf (F)}.
\end{enumerate}
The same statements hold for the category of small maps $(\ct{E}[M], \smallmap{S}[M])$ induced by a set-theoretic model $(M, \epsilon)$ of {\bf RST}.
\end{prop}
\begin{proof}
Routine verification. [We plan to write out at least one of the items in more detail.]
\end{proof}

We derive the promised completeness theorems:
\begin{coro}{complforIZF}
There is a predicative category with class maps $(\ct{E}, \smallmap{S})$ with \smallmap{S} satisfying {\bf (NS), (M)} and {\bf (PS)} such that for the initial \spower-algebra $V$ in $(\ct{E}, \smallmap{S})$ we have
\[ V \models \phi \Leftrightarrow {\bf IZF} \vdash \phi \]
for any set-theoretic sentence $\phi$. Therefore a set-theoretic sentence valid in every initial \spower-algebra in a predicative category with small maps $(\ct{E}, \smallmap{S})$ with \smallmap{S} satisfying {\bf (NS), (M)} and {\bf (PS)} is a consequence of the axioms of {\bf IZF}.
\end{coro}

\begin{coro}{complforCZF}
There is a predicative category with small maps $(\ct{E}, \smallmap{S})$ with \smallmap{S} satisfying {\bf (NS)} and {\bf (F)} such that for the initial \spower-algebra $V$ in $(\ct{E}, \smallmap{S})$ we have 
\[ V \models \phi \Leftrightarrow {\bf CZF} \vdash \phi \]
for any set-theoretic sentence $\phi$. Therefore a set-theoretic sentence valid in every initial \spower-algebra in a predicative category with small maps $(\ct{E}, \smallmap{S})$ with \smallmap{S} satisfying {\bf (NS)} and {\bf (F)} is a consequence of the axioms of {\bf CZF}.
\end{coro}

\begin{rema}{othercomplths}
Completeness theorems of this kind have been proved by various authors, starting with Simpson in \cite{simpson99} (for {\bf IZF}) and Awodey \emph{et al.}~in \cite{awodeyetal04}. Subsequently, predicative versions were proved in \cite{awodeywarren05} and \cite{gambino05}.

Our results improve on these in two respects: firstly, we obtain a completeness theorem for the set theory {\bf CZF}; secondly, we show completeness for both {\bf IZF} and {\bf CZF} with respect to \emph{exact} categories with small maps.
\end{rema}

\appendix

\section{Set-theoretic axioms}

Set theory is a first-order theory with one non-logical binary relation symbol $\epsilon$. Since we are concerned constructive set theories in this paper, the underlying logic will be intuitionistic.

As is customary also in classical set theories like {\bf ZF}, we will use the abbreviations $ \exists x \epsilon a \, (\ldots)$ for $ \exists x \, (x \epsilon a \land \ldots)$, and $\forall x \epsilon a \, (\ldots)$ for $ \forall x \, (x \epsilon a \rightarrow \ldots)$. Recall also that a formula is called \emph{bounded}, when all the quantifiers it contains are of one of these two forms. Finally, a formula of the form $\forall x \epsilon a \, \exists y \epsilon b \, \phi \land \forall y \epsilon b \, \exists x \epsilon a \, \phi$ will be abbreviated as:
\[ \mbox{B}(x \epsilon a, y \epsilon b) \, \phi. \]

Both {\bf IZF} and {\bf CZF} are extensions of the following basic set of axioms, which for convenience we have given a name: {\bf RST}.
\begin{description}
\item[Extensionality:] $\forall x \, ( \, x \epsilon a \leftrightarrow x \epsilon b \, ) \rightarrow a = b$.
\item[Empty set:] $\exists x  \, \forall y  \, \lnot y \epsilon x $.
\item[Pairing:] $\exists x \, \forall y \, (\,  y \epsilon x \leftrightarrow y = a \lor y = b \, )$.
\item[Union:] $\exists x  \, \forall y \, ( \, y \epsilon x \leftrightarrow \exists z \epsilon a \, y \epsilon z  \, )$.
\item[Set induction:] $\forall x \, (\forall y  \epsilon x \, \phi(y) \rightarrow \phi(x)) \rightarrow \forall x \, \phi(x)$.
\item[Bounded separation:] $\exists x \,  \forall y \, ( \, y \epsilon x \leftrightarrow y \epsilon a \land \phi(y) \, ) $, for any bounded formula $\phi$ in which $a$ does not occur.
\item[Strong collection:] $\forall x \epsilon a \, \exists y \, \phi(x,y) \rightarrow \exists b \, \mbox{B}(x \epsilon a, y \epsilon b) \, \phi$.
\end{description}

The intuitionistic set theory {\bf IZF} is obtained by adding to the axioms of {\bf RST} the following:
\begin{description}
\item[Infinity:] $\exists a \, ( \, \exists  x \, x \epsilon a \, ) \land ( \, \forall x  \epsilon a \, \exists y \epsilon a \, x \epsilon y \, )$.
\item[Full separation:] $\exists x \,  \forall y \, ( \, y \epsilon x \leftrightarrow y \epsilon a \land \phi(y) \, ) $, for any formula $\phi$ in which $a$ does not occur.
\item[Power set:] $\exists x \, \forall y \, ( \, y \epsilon x \leftrightarrow y \subseteq a \, )$, where $y \subseteq a$ abbreviates $\forall z \, ( z \epsilon y \rightarrow z \epsilon a)$.
\end{description}

The set theory {\bf CZF}, introduced by Aczel in \cite{aczel78}, is obtained by adding to {\bf RST} the Infinity axiom, as well as a weakening of the Power set axiom called Subset collection:
\begin{description}
\item[Subset collection:] $\exists c \, \forall z \, ( \forall x \epsilon a \, \exists y \epsilon b \, \phi(x,y,z) \rightarrow \exists d \epsilon c \, \mbox{B}(x \epsilon a, y \epsilon d) \, \phi(x, y, z)) $.
\end{description}

\section{Positive Heyting categories}

\begin{defi}{cartesian}
A category \ct{C} is called \emph{cartesian}, when it possesses all finite limits. A functor is \emph{cartesian}, when it preserves finite limits.
\end{defi}

\begin{defi}{regular}
A map \func{f}{B}{A} in a category \ct{C} is called a \emph{cover}, if for any factorisation $f = mg$ in which $m$ is a monomorphism, $m$ is in fact an isomorphism. A cartesian category \ct{C} is called \emph{regular}, when every map factors as a cover followed by a monomorphism, and covers are stable under pullback. A functor is \emph{regular}, when it is cartesian and preserves covers.
\end{defi}

The following lemma about regular categories does not seem to be as well known as it should be:

\begin{lemm}{pspastinregcat}
Consider following commutative diagram
\diag{ A \ar[d] \ar@{->>}[r] &  B \ar[d] \ar[r] & C \ar[d] \\
X \ar@{->>}[r]_{p} & Y \ar[r] & Z}
in a regular category \ct{C}, where $p$ is a cover, as indicated. When the entire diagram is a pullback, and the left-hand square as well, then so is the right-hand square.
\end{lemm}
\begin{proof} See \cite[p.~40]{menni00}.
\end{proof}

\begin{defi}{coherent}
A regular category \ct{C} is called \emph{coherent}, when for each object $X$ in \ct{C} the subobject lattice ${\rm Sub}(X)$ has finite joins, which are, moreover, stable under pulling back along morphisms \func{f}{Y}{X}.
\end{defi}

\begin{defi}{Heyting}
A coherent category \ct{C} is called \emph{Heyting}, when for each morphism \func{f}{Y}{X} the functor
\[ \func{f^*}{{\rm Sub}(X)}{{\rm Sub}(Y)} \]
induced by pullback, has a right adjoint $\forall_f$.
\end{defi}

Heyting categories are rich enough to admit a sound interpretation of  first-order intuitionistic logic. This interpretation of first-order logic is called the \emph{internal logic} of Heyting categories. In this paper, we assume the reader is familiar with this internal logic (if not, see \cite{makkaireyes77}) and frequently exploit it.

\begin{defi}{lextensive}
A cartesian category \ct{C} is called \emph{lextensive} or \emph{positive}, when it has finite sums, which are disjoint and stable .
\end{defi}

Observe that a category that is positive and regular is automatically coherent. Therefore we can axiomatise our basic notion of a positive Heyting category as follows: \ct{E} is a positive Heyting category, when
\begin{enumerate}
\item it is cartesian, i.e., it has finite limits.
\item it is regular, i.e., morphisms factor in a stable fashion as a cover followed by a monomorphism.
\item it is positive, i.e., it has finite sums, which are disjoint and stable.
\item it is Heyting, i.e., for any morphism \func{f}{Y}{X} the induced pullback functor \func{f^*}{{\rm Sub}(X)}{{\rm Sub}(Y)} has a right adjoint $\forall_f$.
\end{enumerate}

\bibliographystyle{plain} \bibliography{ast}

\begin{thebibliography}{10}

\bibitem{aczel78}
P.~Aczel.
\newblock The type theoretic interpretation of constructive set theory.
\newblock In {\em Logic Colloquium '77 (Proc. Conf., Wroc\l aw, 1977)},
  volume~96 of {\em Stud. Logic Foundations Math.}, pages 55--66. North-Holland
  Publishing Co., Amsterdam, 1978.

\bibitem{aczelrathjen01}
P.~Aczel and M.~Rathjen.
\newblock Notes on constructive set theory.
\newblock Technical Report No. 40, Institut Mittag-Leffler, 2000/2001.

\bibitem{awodeyetal04}
S.~Awodey, C.~Butz, A.K. Simpson, and T.~Streicher.
\newblock Relating topos theory and set theory via categories of classes.
\newblock Available from http://www.phil.cmu.edu/projects/ast/, June 2003.

\bibitem{awodeywarren05}
S.~Awodey and M.A. Warren.
\newblock Predicative algebraic set theory.
\newblock {\em Theory Appl. Categ.}, 15:1, 1--39 (electronic), 2005/06.

\bibitem{berg05}
B.~van~den Berg.
\newblock Inductive types and exact completion.
\newblock {\em Ann. Pure Appl. Logic}, 134:95--121, 2005.

\bibitem{berg05a}
B.~van~den Berg.
\newblock Sheaves for predicative toposes.
\newblock Accepted for publication in \emph{Arch. Math. Logic}. Available from
  arXiv: math.LO/0507480, July 2005.

\bibitem{bergdemarchi06}
B.~van~den Berg and F.~De~Marchi.
\newblock Models of non-well-founded sets via an indexed final coalgebra
  theorem.
\newblock Accepted for publication in \emph{J. Symbolic Logic}. Available from
  arXiv: math.LO/0508531, Jan 2006.

\bibitem{bergmoerdijk07c}
B.~van~den Berg and I.~Moerdijk.
\newblock Aspects of predicative algebraic set theory {II}: realisability.
\newblock In preparation, 2007.

\bibitem{bergmoerdijk07d}
B.~van~den Berg and I.~Moerdijk.
\newblock Aspects of predicative algebraic set theory {III}: sheaf models.
\newblock In preparation, 2007.

\bibitem{bergmoerdijk07a}
B.~van~den Berg and I.~Moerdijk.
\newblock A unified approach to algebraic set theory.
\newblock To be published in the proceedings of the Logic Colloquium 2006,
  2007.

\bibitem{blassscedrov89}
A.R. Blass and A.~Scedrov.
\newblock Freyd's models for the independence of the axiom of choice.
\newblock {\em Mem. Am. Math. Soc.}, 79(404), 1989.

\bibitem{fourmanhyland79}
M.~P. Fourman and J.~M.~E. Hyland.
\newblock Sheaf models for analysis.
\newblock In {\em Applications of sheaves (Proc. Res. Sympos. Appl. Sheaf
  Theory to Logic, Algebra and Anal., Univ. Durham, Durham, 1977)}, volume 753
  of {\em Lecture Notes in Math.}, pages 280--301. Springer, Berlin, 1979.

\bibitem{fourmanscedrov82}
M.~P. Fourman and A.~Scedrov.
\newblock The ``world's simplest axiom of choice'' fails.
\newblock {\em Manuscripta Math.}, 38(3):325--332, 1982.

\bibitem{fourman80}
M.P. Fourman.
\newblock Sheaf models for set theory.
\newblock {\em J. Pure Appl. Algebra}, 19:91--101, 1980.

\bibitem{sheaves753}
M.P. Fourman, C.J. Mulvey, and D.S. Scott, editors.
\newblock {\em Applications of sheaves}, volume 753 of {\em Lecture Notes in
  Mathematics}, Berlin, 1979. Springer.

\bibitem{freyd80}
P.J. Freyd.
\newblock The axiom of choice.
\newblock {\em J. Pure Appl. Algebra}, 19:103--125, 1980.

\bibitem{freydfriedmanscedrov87}
P.J. Freyd, H.M. Friedman, and A.~Scedrov.
\newblock Lindenbaum algebras of intuitionistic theories and free categories.
\newblock {\em Ann. Pure Appl. Logic}, 35(2):167--172, 1987.

\bibitem{friedman73}
H.M. Friedman.
\newblock Some applications of {K}leene's methods for intuitionistic systems.
\newblock In {\em Cambridge Summer School in Mathematical Logic (Cambridge,
  1971)}, pages 113--170. Lecture Notes in Math., Vol. 337. Springer, Berlin,
  1973.

\bibitem{gambino07}
N.~Gambino.
\newblock The associated sheaf functor theorem in algebraic set theory.
\newblock To be published in the \emph{Ann. Pure Appl. Logic}.

\bibitem{gambino05}
N.~Gambino.
\newblock Presheaf models for constructive set theories.
\newblock In {\em From sets and types to topology and analysis}, volume~48 of
  {\em Oxford Logic Guides}, pages 62--77. Oxford University Press, Oxford,
  2005.

\bibitem{gambino06}
N.~Gambino.
\newblock Heyting-valued interpretations for constructive set theory.
\newblock {\em Ann. Pure Appl. Logic}, 137(1-3):164--188, 2006.

\bibitem{hyland82}
J.M.E. Hyland.
\newblock The effective topos.
\newblock In {\em The L.E.J. Brouwer Centenary Symposium (Noordwijkerhout,
  1981)}, volume 110 of {\em Stud. Logic Foundations Math.}, pages 165--216.
  North-Holland Publishing Co., Amsterdam, 1982.

\bibitem{hylandpitts89}
J.M.E. Hyland and A.M. Pitts.
\newblock The theory of constructions: categorical semantics and
  topos-theoretic models.
\newblock In {\em Categories in computer science and logic (Boulder, CO,
  1987)}, volume~92 of {\em Contemp. Math.}, pages 137--199. Amer. Math. Soc.,
  Providence, RI, 1989.

\bibitem{joyalmoerdijk94}
A.~Joyal and I.~Moerdijk.
\newblock A completeness theorem for open maps.
\newblock {\em Annals of Pure and Applied Logic}, 70:51--86, 1994.

\bibitem{joyalmoerdijk95}
A.~Joyal and I.~Moerdijk.
\newblock {\em Algebraic set theory}, volume 220 of {\em London Mathematical
  Society Lecture Note Series}.
\newblock Cambridge University Press, Cambridge, 1995.

\bibitem{kouwenhovenvanoosten05}
C.~Kouwenhoven-Gentil and J.~van Oosten.
\newblock Algebraic set theory and the effective topos.
\newblock {\em J. Symbolic Logic}, 70(3):879--890, 2005.

\bibitem{lackvitale01}
S.~Lack and E.M. Vitale.
\newblock When do completion processes give rise to extensive categories?
\newblock {\em J. Pure Appl. Algebra}, 159(2-3):203--230, 2001.

\bibitem{lambek70}
J.~Lambek.
\newblock Subequalizers.
\newblock {\em Canadian Math. Bull.}, 13:337--349, 1970.

\bibitem{lambekscott86}
J.~Lambek and P.J. Scott.
\newblock {\em Introduction to higher order categorical logic}, volume~7 of
  {\em Camb. Stud. Adv. Math.}
\newblock Cambridge University Press, Cambridge, 1986.

\bibitem{lubarsky06}
R.S. Lubarsky.
\newblock {CZF} and {S}econd {O}rder {A}rithmetic.
\newblock {\em Ann. Pure Appl. Logic}, 141(1-2):29--34, 2006.

\bibitem{maclanemoerdijk92}
S.~Mac~Lane and I.~Moerdijk.
\newblock {\em Sheaves in geometry and logic -- A first introduction to topos
  theory}.
\newblock Universitext. Springer-Verlag, New York, 1992.

\bibitem{makkaireyes77}
M.~Makkai and G.E. Reyes.
\newblock {\em First order categorical logic}.
\newblock Springer-Verlag, Berlin, 1977.
\newblock Model-theoretical methods in the theory of topoi and related
  categories, Lecture Notes in Mathematics, Vol. 611.

\bibitem{martinlof84}
P.~Martin-L{\"o}f.
\newblock {\em Intuitionistic type theory}, volume~1 of {\em Studies in Proof
  Theory. Lecture Notes}.
\newblock Bibliopolis, Naples, 1984.

\bibitem{mccarty84}
D.C. McCarty.
\newblock {\em Realizability and recursive mathematics}.
\newblock PhD thesis, Oxford University, 1984.

\bibitem{menni00}
M.~Menni.
\newblock {\em Exact completions and toposes}.
\newblock PhD thesis, University of Edinburgh, 2000.
\newblock Available at
  http://www.math.uu.nl/people/jvoosten/realizability.html.

\bibitem{moerdijkpalmgren00}
I.~Moerdijk and E.~Palmgren.
\newblock Wellfounded trees in categories.
\newblock {\em Ann. Pure Appl. Logic}, 104(1-3):189--218, 2000.

\bibitem{moerdijkpalmgren02}
I.~Moerdijk and E.~Palmgren.
\newblock Type theories, toposes and constructive set theory: predicative
  aspects of {AST}.
\newblock {\em Ann. Pure Appl. Logic}, 114(1-3):155--201, 2002.

\bibitem{rathjen06}
M.~Rathjen.
\newblock Realizability for constructive {Z}ermelo-{F}raenkel set theory.
\newblock In {\em Logic Colloquium '03}, volume~24 of {\em Lect. Notes Log.},
  pages 282--314. Assoc. Symbol. Logic, La Jolla, CA, 2006.

\bibitem{simpson99}
A.K. Simpson.
\newblock Elementary axioms for categories of classes (extended abstract).
\newblock In {\em 14th Symposium on Logic in Computer Science (Trento, 1999)},
  pages 77--85. IEEE Computer Soc., Los Alamitos, CA, 1999.

\bibitem{streicher05}
T.~Streicher.
\newblock Realizability models for {CZF}+ $\lnot$ {P}ow.
\newblock Unpublished note, March 2005.

\bibitem{warren07}
M.A. Warren.
\newblock Coalgebras in a category of classes.
\newblock {\em Ann. Pure Appl. Logic}, 146(1):60--71, 2007.

\bibitem{wraith74}
G.~Wraith.
\newblock Artin glueing.
\newblock {\em J. Pure Appl. Algebra}, 4:345--348, 1974.

\end{thebibliography}

\end{document}